\documentclass[12pt,reqno]{amsart}
\usepackage{amsmath}
\usepackage{amssymb}
\usepackage{tikz}
\usepackage{bbm}

\newtheorem{theorem}{Theorem}
\newtheorem{prop}{Proposition}
\newtheorem{lemma}{Lemma}
\newtheorem{cor}{Corollary}

\theoremstyle{definition}

\newtheorem{defn}{Definition}

\newtheorem{remark}{Remark}
\newtheorem{notation}{Notation}

\def\A{\mathcal A}

\def\m{\mathcal M}

\def\N{\mathbb N}
\def\R{\mathbb R}

\def\phi{\varphi}
\def\rho{\varrho}
\def\epsilon{\varepsilon}
\def\mm{\mathfrak C}
\def\nn{\mathfrak D}
\def\ee{\mathfrak E}
\def\ss{\mathcal I}
\def\t{\mathcal T}
\def\sm{\mathcal S}

\def\na{\mathbb M}

\begin{document}

\newcounter{algnum}
\newcounter{step}
\newtheorem{alg}{Algorithm}

\newenvironment{algorithm}{\begin{alg}\end{alg}}

\title[Joint spectral radius, Sturmian measures, finiteness conjecture]
{Joint spectral radius, Sturmian measures, and the finiteness conjecture}

\author{O.~Jenkinson \& M. Pollicott}

\address{Oliver Jenkinson;
School of Mathematical Sciences, Queen
  Mary, University of London, Mile End Road, London, E1 4NS, UK.
\newline {\tt omj@maths.qmul.ac.uk}}

\address{Mark Pollicott;
Mathematics Institute,
University of Warwick,
Coventry, CV4 7AL, UK.
\newline {\tt mpollic@maths.warwick.ac.uk}}


\begin{abstract}
The joint spectral radius of a pair of $2 \times 2$ real matrices $(A_0,A_1)\in M_2(\mathbb{R})^2$
is defined  to be
$r(A_0,A_1)= \limsup_{n\to\infty} \max \{\|A_{i_1}\cdots A_{i_n}\|^{1/n}: i_j\in\{0,1\}\}$,
the optimal growth rate of  the norm of products of these matrices.

The Lagarias-Wang finiteness conjecture \cite{lagariaswang}, asserting that $r(A_0,A_1)$ is
always the $n$th root of the spectral radius of some length-$n$ product $A_{i_1}\cdots A_{i_n}$, has been refuted by
Bousch \& Mairesse \cite{bouschmairesse}, with subsequent counterexamples presented by
Blondel, Theys \& Vladimirov \cite{btv}, Kozyakin \cite{kozyakin},  Hare, Morris, Sidorov \& Theys
\cite{hmst}.

In this article we introduce a new approach to generating finiteness counterexamples, 
and use this to exhibit an open subset of $M_2(\mathbb{R})^2$ with the property that each member $(A_0,A_1)$ of the subset generates 
uncountably many counterexamples of the form $(A_0,tA_1)$.
Our methods employ ergodic theory, in particular the analysis of Sturmian invariant measures; 
 this approach allows a short proof 
that the relation between the parameter $t$ and the Sturmian parameter $\mathcal{P}(t)$ is a devil's staircase.
\end{abstract}

\maketitle

\section{Introduction}\label{generalsection}

\subsection{Problem and setting}\label{problemsetting}

For a square matrix $A$ with real entries, its \emph{spectral radius} $r(A)$, defined as the maximum modulus of its eigenvalues, satisfies \emph{Gelfand's formula}
$$
r(A)= \lim_{n\to\infty} \|A^n\|^{1/n}\,,
$$
where $\|\cdot\|$ is a matrix norm.
More generally, for a finite collection $\A=\{A_0,\ldots, A_l\}$ of real square matrices, all of the same size,
the \emph{joint spectral radius} $r(\A)$ is defined by
\begin{equation}\label{jsrlimsup}
r(\A) = 
\limsup_{n\to\infty} \max \{\|A_{i_1}\cdots A_{i_n}\|^{1/n}: i_j\in\{0,\ldots, l\}\} \,,
\end{equation}
or equivalently (see e.g.~\cite{jungers}) by
\begin{equation}\label{rjsr}
r(\A) = 
 \lim_{n \to +\infty} \max
\{  r(A_{i_1}\cdots A_{i_n})^{1/n}: i_j\in\{0, \dots,l\}\} \,.
\end{equation}
The notion of joint spectral radius was introduced by Rota \& Strang \cite{rotastrang},
and notably popularised by Daubechies \& Lagarias \cite{daubechieslagarias} in their work on wavelets.
Since the 1990s it has become an area of very active research interest, from both a pure and an applied perspective
(see e.g.~\cite{blondel, jungers, kozyakinbiblio, strang}).

The set $\A$ is said to have the \emph{finiteness property} if $r(\A)=r(A_{i_1}\cdots A_{i_n})^{1/n}$
for some $i_1,\ldots, i_n\in\{0,\ldots, l\}$.
It was conjectured by Lagarias \& Wang \cite{lagariaswang}
(see also Gurvits \cite{gurvits}) that every such $\A$ enjoys the finiteness property.
This so-called finiteness conjecture was, however, refuted by Bousch \& Mairesse \cite{bouschmairesse}, and a number of authors (see \cite{btv,hmst, kozyakin, morrissidorov}) have subsequently given examples of sets $\A$ for which the finiteness property fails.
A common feature of these finiteness counterexamples has been a judicious choice of a pair of $2 \times 2$ matrices $A_0,A_1$, followed by an argument that for certain $t>0$, the finiteness property fails for the set $\A(t)= \{ A_0^{(t)}, A_1^{(t)}\} = \{A_0, tA_1\}$. 

In fact for many of these examples it has been observed that the family $(\A(t))_{t>0}$ can be associated with the class of \emph{Sturmian} sequences of Morse \& Hedlund \cite{morsehedlund}: 
for a given $t>0$ an appropriate Sturmian sequence $(i_n)_{n=1}^\infty \in\{0,1\}^\N$ turns out to give the optimal matrix product, in the sense that the joint spectral radius $r(\A(t))$ equals
$\lim_{n\to\infty} r(A_{i_1}^{(t)}\cdots A_{i_n}^{(t)})^{1/n}$ 
(see \cite{btv,bouschmairesse,hmst,kozyakin,morrissidorov} for further details).
A Sturmian sequence $(i_n)_{n=1}^\infty$ has a well-defined \emph{1-frequency}
$\mathcal{P} = \lim_{N\to\infty} \frac{1}{N} \sum_{n=1}^N i_n$, and it is those sets $\A(t)$ whose associated Sturmian sequences\footnote{We follow the definition of Sturmian sequence given in \cite{bullettsentenac}, though note that
some authors refer to these as \emph{balanced} sequences,
reserving the nomenclature \emph{Sturmian} precisely for those balanced sequences with irrational 1-frequency.}
 have \emph{irrational} 1-frequency which yield counterexamples to the finiteness conjecture
 (see Proposition \ref{irrationalcounterexampleconnection} below for a more precise description of the connection between finiteness counterexamples and Sturmian sequences with irrational 1-frequency).
For certain such families $(\A(t))_{t>0}$ (which henceforth we refer to as \emph{Sturmian families}), it has been proved by Morris \& Sidorov \cite{morrissidorov} 
(see also \cite[p.~109]{bouschmairesse})
that
if $\mathcal{P}(t)$ denotes the 1-frequency associated 
to $\A(t)$, then the parameter
mapping $t\mapsto \mathcal{P}(t)$ is continuous and monotone, but \emph{singular}
in the sense that 
$\{t>0:\mathcal{P}(t)\notin \mathbb{Q}\}$ is nowhere dense; in other words, 
the uncountably many parameters
$t$ for which finiteness counterexamples occur only constitute a thin subset\footnote{The belief that finiteness counterexamples are rare appears to be widespread; for example Maesumi \cite{maesumi} conjectures
that they consitute a set of (Lebesgue) measure zero in the space of matrices.} 
of $\mathbb{R}^+$.

Examples of Sturmian families $(\A(t))_{t>0}$
have been given by
Bousch \& Mairesse \cite{bouschmairesse}, who considered the family
generated by matrix pairs of the form
\begin{equation}\label{bmfamily}
\A= \left( \begin{pmatrix} e^{\kappa h_0}+1 & 0 \cr e^\kappa & 1 \cr \end{pmatrix}\, , \
\begin{pmatrix} 1 & e^\kappa \cr 0 & e^{\kappa h_1} +1 \cr \end{pmatrix} \right)\,,\,
\kappa>0\,, h_0,h_1>0\,, h_0+h_1 < 2\,,
\end{equation}
by Kozyakin \cite{kozyakin}, who studied the family
generated by pairs of the form
\begin{equation}\label{kozyakinpairs}
\A = \left(  \begin{pmatrix} 1 & 0 \cr c & d \cr \end{pmatrix}\, , \
 \begin{pmatrix} a & b \cr 0 & 1 \cr \end{pmatrix}  \right) \,,\,
 0<a,d<1\le bc\,,
\end{equation}
and by various authors \cite{btv,hmst,morrissidorov} focusing on the family generated by
the particular pair 
\begin{equation}\label{standardpair}
\A = \left(  \begin{pmatrix} 1 & 0 \cr 1 & 1 \cr \end{pmatrix}\, , \
 \begin{pmatrix} 1 & 1 \cr 0 & 1 \cr \end{pmatrix}  \right) \,.
\end{equation}

 
 For an invertible matrix $P$, 
 the simultaneous similarity $(A_0,A_1)\mapsto (P^{-1}A_0P, P^{-1}A_1P)$
leaves invariant the joint spectral radius, 
and does not change the sequences $(i_n)_{n=1}^\infty$ attaining the optimal matrix product,
while if $u,v>0$ then $(uA_0,vA_1)$ 
has the same optimizing sequences as $(A_0, (v/u)A_1)$.
Therefore, declaring $\A=(A_0,A_1)$ and $\A'=(A_0',A_1')$ to be equivalent if
$A_0'=uP^{-1}A_0P$ and $A_1'=vP^{-1}A_1P$ for some invertible $P$ and $u,v>0$, we see that the
equivalence of $\A$ and $\A'$ implies that $(\A(t))_{t>0}$ is a Sturmian family if and only if
$(\A'(t))_{t>0}$ is.
In particular, $(\A'(t))_{t>0}$ is a Sturmian family whenever $\A'$ is equivalent to a matrix pair $\A$ of the form
(\ref{bmfamily}), (\ref{kozyakinpairs}), or (\ref{standardpair}).

The purpose of this article is to introduce an approach to studying the joint spectral radius and
generating finiteness counterexamples, which in particular yields new examples of Sturmian
families $(\A(t))_{t>0}$, i.e.~where $\A$ is not equivalent to a matrix pair of the form
(\ref{bmfamily}), (\ref{kozyakinpairs}), or (\ref{standardpair}).
 Our method is conceptually different to previous authors, employing notions from
 dynamical systems, ergodic theory, and in particular ergodic optimization (see e.g.~\cite{jeo}).
Specifically, we identify a dynamical system $T_\A$ with the matrix pair $\A=(A_0,A_1)$,
and cast the problem of determining the joint spectral radius $r(\A)$ in terms of ergodic optimization
(see Theorem \ref{maxtheoremintro} below): it suffices to determine the
$T_\A$-invariant probability measure which maximizes the integral of a certain auxiliary real-valued function $f_\A$.
Working with the family of \emph{$\A$-Sturmian measures} (certain probability measures
invariant under $T_\A$) instead of Sturmian sequences, we exploit a characterisation of these measures in terms of the smallness of their support to show that they give precisely the family
of $f_{\A(t)}$-maximizing measures, $t>0$.
In particular, whenever the $f_{\A(t)}$-maximizing measure is Sturmian of \emph{irrational} parameter
$\mathcal{P}(t)$ then $\A(t)$ is a finiteness counterexample (cf.~Proposition \ref{irrationalcounterexampleconnection}).

The $\A$-Sturmian measures are naturally identified with Sturmian measures on $\Omega = \{0,1\}^\N$, the full shift on two symbols (see Notation \ref{sturmianfullshift}). A notable feature of our approach is that the singularity of the parameter mapping
$t\mapsto \mathcal{P}(t)$ (and in particular the fact that $\{t>0:\mathcal{P}(t)\notin\mathbb{Q}\}$ is nowhere dense in $\mathbb{R}^+$)
is then readily deduced (see Theorem \ref{deviltheorem} in  \S \ref{devilsection}) as a consequence of classical facts about 
parameter dependence of Sturmian measures on $\Omega$
(i.e.~rather than requiring the {\it ab initio} approach of \cite{morrissidorov}).

\subsection{Statement of results}\label{statementsubsection}

We use $M_2(\mathbb{R})$ to denote the set of real $2 \times 2$ matrices,
and focus attention on certain of its open subsets:

\begin{notation} 
$M_2(\mathbb{R}^+)$  will denote the set of \emph{positive matrices}, i.e.~matrices in $M_2(\mathbb{R})$
with entries in $\mathbb{R}^+ =\{x\in\mathbb{R}:x>0\}$, 
and
 $M_2^+(\mathbb{R}^+)=\{A\in M_2(\mathbb{R}^+):\det A>0\}$
 will denote the set of \emph{positive orientation-preserving matrices}.
 \end{notation}

Turning to \emph{pairs} of matrices, we shall consider the following open subset
of $M_2^+(\mathbb{R}^+)^2$:

\begin{defn}\label{frakmdefn}
Let $\mm\subset M_2^+(\mathbb{R}^+)^2$ denote the set of matrix pairs
\begin{gather*}
\left( A_0, A_1\right)
= \left( \begin{pmatrix} a_0&b_0\\c_0&d_0\end{pmatrix} ,
\begin{pmatrix} a_1&b_1\\ c_1& d_1\end{pmatrix} \right) \in M_2^+(\mathbb{R}^+)^2
\end{gather*}
satisfying 
\begin{equation}\label{definequal1}
\frac{a_0}{c_0}  <  \frac{b_1}{d_1}
\end{equation}
and 
\begin{equation}\label{definequal2}
a_1+c_1-b_1-d_1 < 0 < a_0+c_0-b_0-d_0\,.
\end{equation}
For reasons which will become apparent later (see Proposition \ref{concavconvexprop}),
$\mm$ will be referred to as the set of \emph{concave-convex matrix pairs}.
\end{defn}

Finally, our counterexamples to the Lagarias-Wang finiteness conjecture will be drawn from 
a certain open subset  $\nn$
(given by Definition \ref{defnofnn} below) of $\mm$
which is conveniently described in terms of quantities $\rho_A$ and $\sigma_A$ defined as follows:

\begin{defn}
For $A= \begin{pmatrix} a&b\\c&d\end{pmatrix} \in M_2^+(\mathbb{R}^+)$,
define
$$
\rho_A = \frac{2b}{a-d-2b+\sqrt{(a-d)^2+4bc}}\,,
$$
and if $a+c\neq b+d$ then define
$$
\sigma_A = \frac{b-a}{a+c-b-d}\,.
$$
\end{defn}

It turns out (see Corollary \ref{rhoposlessminusone}) that if
$(A_0,A_1)\in\mm$
then
$\sigma_{A_0}<0<\rho_{A_0}$ and $\rho_{A_1}<-1$.
The set $\nn$ is defined by imposing two inequalities:

\begin{defn}\label{defnofnn}
Define 
\begin{equation*}
\nn
=
\left\{\, (A_0,A_1)\in\mm : \
\rho_{A_1} < \sigma_{A_0}
\ \text{and}\
\sigma_{A_1} < \rho_{A_0} \right\}\,.
\end{equation*}
\end{defn}

Clearly $\nn$ is an open subset of $\mm$, hence also of $M_2(\mathbb{R})^2$.
It is also non-empty: for example 
it is readily verified that the two-parameter family
\begin{equation}\label{2dimreg}
\nn_2 = \left\{
\left( \begin{pmatrix} 1&b\\c&1\end{pmatrix} ,
\begin{pmatrix} 1& c\\  b & 1\end{pmatrix} \right)
: \
bc < 1 < c \ \ , \ \ (b,c)\in (\mathbb{R}^+)^2
\right\}
\end{equation}
is a subset of $\nn$.
%
%
%
Note that the pair (\ref{standardpair})
studied in \cite{btv, hmst,morrissidorov}, and
corresponding to $(b,c)=(0,1)$ in (\ref{2dimreg}), lies on the boundary of both $\nn$ and $\nn_2$.

\medskip

A version of our main result is the following:

\begin{theorem}\label{nontechnictheorem}
The open subset $\nn\subset M_2(\mathbb{R})^2$ is such that if $\A=(A_0,A_1)\in\nn$ then
for uncountably many $t\in\mathbb{R}^+$, the matrix
pair $(A_0,tA_1)$
is a
finiteness counterexample.
\end{theorem}

If we define $\ee \subset M_2(\R)^2$
to be the set of matrix pairs which are equivalent to some pair in $\nn$
(recall that
$\A=(A_0,A_1)$ and $\A'=(A_0',A_1')$ are equivalent if
$A_0'=uP^{-1}A_0P$ and $A_1'=vP^{-1}A_1P$ for some invertible $P$ and $u,v>0$) then clearly:

\begin{cor}\label{nontechnictheoremcor}
The open subset $\ee\subset M_2(\mathbb{R})^2$ is such that if $\A=(A_0,A_1)\in\ee$ then
for uncountably many $t\in\mathbb{R}^+$, the matrix
pair $(A_0,tA_1)$
is a
finiteness counterexample.
\end{cor}

\begin{remark}
Theorem \ref{nontechnictheorem} yields new finiteness counterexamples, in the sense that $\nn$ contains matrix pairs 
which are not equivalent to pairs satisfying (\ref{bmfamily}), (\ref{kozyakinpairs}), or (\ref{standardpair}).
To see this, note for example that
\begin{equation}\label{exampleofa}
\A =  (A_0,A_1) = 
\left(  \begin{pmatrix} 5/8 & 3/112 \cr 7/8 & 15/16 \cr \end{pmatrix}\, , \
 \begin{pmatrix} 15/16 & 1 \cr 1/128 & 7/8 \cr \end{pmatrix}  \right) 
\end{equation}
belongs to $\nn$.
Both $A_0$ and $A_1$ have their larger eigenvalue equal to 1, and smaller eigenvalues given by
$\lambda_0=9/16$ and $\lambda_1=13/16$, respectively.
Now both matrices in (\ref{standardpair}) have the single eigenvalue 1, so (\ref{exampleofa}) cannot be equivalent to (\ref{standardpair}); moreover (\ref{exampleofa}) is not equivalent to any matrix pair satisfying (\ref{bmfamily}), since both matrices in (\ref{bmfamily}) have the property that the larger eigenvalue is more than double the smaller eigenvalue.
Lastly, we show that $\A$ in (\ref{exampleofa}) is not equivalent to any pair $\A'$ satisfying (\ref{kozyakinpairs}),
i.e.~$\A'=(A_0',A_1')$ with
$A_0' = \begin{pmatrix} 1 & 0 \cr c & d \cr \end{pmatrix}$,
$A_1'= \begin{pmatrix} a & b \cr 0 & 1 \cr \end{pmatrix}$,
where  $0<a,d<1\le bc$.
Note that both $A_0'$ and $A_1'$ have their larger eigenvalue equal to 1 (as is the case for $A_0$ and $A_1$),
and smaller eigenvalues equal to $d$ and $a$, respectively.
Thus if $\A$ and $\A'$ were equivalent then there would exist an invertible $P$ such that
$A_0'=P^{-1}A_0P$ and $A_1'=P^{-1}A_1P$ (i.e.~the positive reals $u,v$ in the above definition of equivalence must both equal 1), so that $d=\lambda_0=9/16$, $a=\lambda_1=13/16$, and $\text{trace}(A_0A_1)=\text{trace}(A_0'A_1')$.
In particular, 
$1 \le bc = \text{trace}(A_0'A_1')-d-a = \text{trace}(A_0A_1)-\lambda_0-\lambda_1 = 12995/14336 <1$, a contradiction.
It follows that $\A\in\nn$ given by (\ref{exampleofa}) is not equivalent to any matrix pair satisfying
(\ref{kozyakinpairs}).
\end{remark}

%
%

A key tool in proving Theorem \ref{nontechnictheorem} is the following Theorem \ref{maxtheoremintro}
(proved in \S \ref{induceddynsyssection} as Theorem \ref{maxtheorem})
characterising the joint spectral radius of $\A\in\mm$ in terms of maximizing the integral of a certain
function $f_\A$ over the set $\m_\A$ of probability measures invariant under an associated mapping $T_\A$.
More precisely, the action of any positive matrix $A$ on $(\mathbb{R}^+)^2$ induces a projective map $T_A$ (see \S \ref{inducproj}),
and if $\A=(A_0,A_1)\in\mm$
then the inverses $T_{A_0}^{-1}$, $T_{A_1}^{-1}$
together define a two-branch
dynamical system $T_\A$ (see \S \ref{induceddynsyssection}) on a subset of the unit interval $X$.
Defining the real-valued function $f_\A$, 
in terms of the derivative $T_\A'$ and characteristic functions of the images $T_{A_0}(X)$ and $T_{A_1}(X)$,
by
\begin{equation*}
f _\A =
\frac{1}{2}
\left(
\log T_\A' +(\log \det A_0)\mathbbm{1}_{T_{A_0}(X)}   +(\log \det A_1)\mathbbm{1}_{T_{A_1}(X)}
\right)
\end{equation*}
then gives:

\begin{theorem}\label{maxtheoremintro}
If $\A\in\mm$ then
\begin{equation}\label{maxtheoremintroeq}
\log r(\A) = 
\max_{\mu\in\m_\A} \int f_\A\, d\mu  \,.
\end{equation}
\end{theorem}

\medskip

In order to state a more precise version of Theorem \ref{nontechnictheorem}, 
we first need
some
basic facts concerning ergodic theory, symbolic dynamics, and Sturmian measures:

\begin{notation}\label{sturmianfullshift}
Let $\Omega = \{0,1\}^\N$ denote the set of one-sided sequences $\omega=(\omega_n)_{n=1}^\infty$,
where $\omega_n\in\{0,1\}$ for all $n\ge1$.
When equipped with the product topology, $\Omega$ becomes a compact space,
and the shift map $\sigma:\Omega\to\Omega$ defined by
$(\sigma\omega)_n = \omega_{n+1}$ for all $n\ge1$ is then continuous.
Let $\m$ denote the set of shift-invariant Borel probability measures
on $\Omega$; when equipped with the weak-$*$ topology $\m$ is compact (see \cite[Thm.~6.10]{walters}).

We equip $\Omega$ with the lexicographic order $<$,
and write $[\omega^-,\omega^+]=\{\omega\in\Omega: \omega^-\le \omega \le \omega^+\}$.
A \emph{Sturmian interval} is one of the form $[0\omega,1\omega]$, for some $\omega\in\Omega$.
A measure $\mu\in\m$ is called \emph{Sturmian} (see e.g.~\cite[Prop~1.5]{bouschmairesse}, \cite{bullettsentenac})
if its support is contained in a Sturmian interval.
Let $\mathcal{S}\subset \m$ denote the class of Sturmian measures on $\Omega$.
For a Sturmian measure $\mu\in\mathcal{S}$, the value
$\mu([1])$, denoted $\mathcal{P}(\mu)$, is called its \emph{(Sturmian) parameter}\footnote{This 
corresponds to the \emph{1-frequency} mentioned in \S \ref{problemsetting}, sometimes called
the \emph{1-ratio} (see e.g.~\cite{hmst, morrissidorov}), 
or the \emph{rotation number} (see e.g.~\cite{bullettsentenac}).}, where $[1]$ denotes the (cylinder) set  
$\{\omega\in\Omega : \omega_1=1\}$.
A \emph{Sturmian sequence} of parameter $\mathcal{P}$ is any point in the support of the Sturmian measure of parameter $\mathcal{P}$.
\end{notation}

The following are classical facts about Sturmian measures (see e.g.~\cite[\S 1.1]{bouschmairesse} or \cite{bullettsentenac}):

\begin{prop}\label{sturmianclassical}
\item[\, (a)]
For each Sturmian interval $[0\omega,1\omega]\subset\Omega$ there exists a unique Sturmian measure
whose support is contained in this interval.
\item[\, (b)]
The mapping $\mathcal{P}:\mathcal{S}\to[0,1]$ is a homeomorphism.
If $\mu\in\mathcal{S}$ has $\mathcal{P}(\mu)\in\mathbb{Q}$ then its support is a single $\sigma$-periodic orbit, while if $\mathcal{P}(\mu)\notin\mathbb{Q}$ then its support
is a Cantor subset of $\Omega$ which supports no other $\sigma$-invariant measure (and in particular contains no periodic orbit).
\item[\, (c)]
If $d(\omega)$ denotes the Sturmian parameter of the Sturmian
measure supported by the Sturmian interval $[0\omega,1\omega]\subset \Omega$,
then the map $d:\Omega\to [0,1]$ is 
continuous, non-decreasing, and surjective. 
The preimage $d^{-1}(\mathcal{P})$ is a singleton if $\mathcal{P}$ is irrational,
and a positive-length closed interval if $\mathcal{P}$ is rational.
\end{prop}

For example the Sturmian measures of parameter
$1/2$, $1/3$, $2/5$, $3/8$ and $5/13$ are, respectively,
supported by the $\sigma$-periodic orbits generated by the finite words
$$
01
\, ,\
001
\, ,\
00101
\, ,\
00100101
\, ,\
0010010100101
\,,
$$
whereas the Sturmian measure of parameter
$(3-\sqrt{5})/2$ is 
supported by
the smallest Cantor set
containing the $\sigma$-orbit of 
$$
0010010100100101001010010010100101
\ldots
$$

In view of Theorem \ref{maxtheoremintro}, for a matrix pair $\A\in\mm$ we are interested in
measures $\nu\in\m_\A$ attaining the maximum in (\ref{maxtheoremintroeq}),
i.e.~satisfying $\int f_\A\, d\nu = \max_{\mu\in\m_\A} \int f_\A\, d\mu$;
such $\nu$ will be called \emph{$f_\A$-maximizing}. 
There is a topological conjugacy between $T_\A$ and the shift map
$\sigma:\Omega\to\Omega$, 
and this induces a natural homeomorphism between $\m_\A$ and $\m$;
the image of any $f_\A$-maximizing measure under this homeomorphism will be 
called a \emph{maximizing measure for $\A$}.
We then say that $\A=(A_0,A_1)\in \mm$ \emph{generates a full Sturmian family}
if the set of maximizing measures for the family $\A(t) = (A_0,tA_1)$, $t\in\mathbb{R}^+$,
is precisely the set $\mathcal{S}$ of all Sturmian measures on $\Omega$.

A more precise version of our main result Theorem \ref{nontechnictheorem} is then the following:

\begin{theorem}\label{maintheorem}
Every matrix pair in the open subset $\nn\subset M_2(\mathbb{R})^2$ 
(and hence the open subset $\ee\subset M_2(\mathbb{R})^2$)
generates a full Sturmian family.
\end{theorem}


Note that Theorem \ref{maintheorem} will follow from a more detailed version,
Theorem \ref{deviltheorem}, which in particular incorporates the 
statement that the parameter map $t\mapsto \mathcal{P}(t)$ is a devil's staircase.

\subsection{Relation with previous results}\label{relatsubsection}

The methods of this paper can also be used to give an alternative proof of some of the results mentioned above,
namely establishing the analogue of Theorem \ref{maintheorem} in certain cases treated by
Bousch \& Mairesse \cite{bouschmairesse} and Kozyakin \cite{kozyakin},
 and the case considered by
 Blondel, Theys \& Vladimirov \cite{btv}, Hare, Morris, Sidorov \& Theys \cite{hmst}, and Morris \& Sidorov \cite{morrissidorov}.

As already noted, the matrix pair (\ref{standardpair}) lies on the boundary of our open set $\nn$,
and clearly it also lies on the boundary of the set $\mathfrak{K}\subset M_2^+(\mathbb{R}^+)^2$ defined by 
Kozyakin's conditions (\ref{kozyakinpairs}).
It can be checked that $\mathfrak{K}$ itself lies in the boundary of our set $\mm$, but not in the boundary of $\nn$.
However, the subset $\mathfrak{K}' \subset \mathfrak{K}$ defined by
\begin{equation}\label{kprime}
\mathfrak{K}'
=
\left\{
\left(  \begin{pmatrix} 1 & 0 \cr c & d \cr \end{pmatrix}\, , \
 \begin{pmatrix} a & b \cr 0 & 1 \cr \end{pmatrix}  \right) \in\mathfrak{K}
 :
 a\le b\ \text{and}\ d\le c\right\} \,,
\end{equation}
can be readily checked to lie in the boundary of $\nn$.
Matrices in the Bousch-Mairesse family (\ref{bmfamily}) do not all satisfy our condition
(\ref{definequal1}), or indeed the corresponding weak inequality, 
so do not automatically belong to the boundary of $\mm$.
However, imposing the additional condition
\begin{equation}\label{bmadditional}
e^{2\kappa} \ge (e^{\kappa h_0}+1)(e^{\kappa h_1}+1)
\end{equation}
ensures that a matrix pair satisfying (\ref{bmfamily}) belongs to the boundary
of $\mm$, and indeed also belongs to the boundary of $\nn$.
In \S \ref{adaptations} we will indicate the minor modifications to our approach needed to
handle the case of (\ref{standardpair}), and the sub-cases of (\ref{bmfamily}) and (\ref{kozyakinpairs})
defined by (\ref{bmadditional}) and (\ref{kprime}) respectively.

\subsection{Organisation of article}
The article is organised as follows. Section \ref{prelimsection} consists of preliminaries:
 maps induced by matrices acting on projective space, Perron-Frobenius theory, and some useful notation
 and identities. Section \ref{projconvprojconc} develops the notions of projective convexity and projective concavity.
 Section \ref{induceddynsyssection} introduces the induced dynamical system $T_\A$ for concave-convex matrix pairs $\A$,
 the formulation of joint spectral radius in terms of ergodic optimization (Theorem \ref{maxtheorem}),
 and the connection between the finiteness property and $T_\A$-periodic orbits.
 Section \ref{sturmasection} introduces Sturmian measures and  Sturmian intervals for the dynamical system $T_\A$,
 and makes the connection between finiteness counterexamples and unique maximizing measures which are Sturmian of irrational parameter.
Section \ref{transfersection} establishes the existence of an important technical tool, the Sturmian transfer function.
After deriving some explicit formulae for extremal Sturmian intervals in Section \ref{extsturmsection},
the key Section \ref{particularsection} establishes the link between 
Sturmian intervals and the parameter $t$ of the pair $\A(t)$.
Section \ref{dominatessec} treats the case of those parameters $t$ such that one matrix in the pair $\A(t)$
dominates the other, so that the joint spectral radius $r(\A(t))$ is simply the spectral radius of the dominating matrix.
All other parameters are considered in Section \ref{technicalsection}, establishing that the joint spectral radius
is always  attained by a unique Sturmian measure.
Finally, in Section \ref{devilsection} we show that the map taking parameter values $t$ to the associated Sturmian
parameter
 $\mathcal{P}(t)$ is a devil's staircase.

\section{Preliminaries}\label{prelimsection}

\subsection{The induced map for a positive matrix}\label{inducproj}

\begin{notation}
Throughout we use the notation $X=[0,1]$.
\end{notation}

A positive matrix $A\in M_2(\mathbb{R}^+)$ gives a self-map 
$v \mapsto Av$
of $(\mathbb{R}^+)^2$.
This lifts to a self-map $\widetilde A: [v]\mapsto [Av]$
of projective space $(\mathbb{R}^+)^2/\sim$,  the equivalence relation $\sim$ being defined by
$v\sim v'$ if $v=sv'$ for some $s>0$, and $[v]$ denoting the equivalence class containing $v\in(\mathbb{R}^+)^2$.
It is convenient to identify projective space with 
\begin{equation*}
\Sigma=\left\{ \begin{pmatrix} x\\ 1-x\end{pmatrix} : x\in (0,1)\right\} \,,
\end{equation*}
so that the projection $\pi:(\mathbb{R}^+)^2 \to \Sigma$ takes the form
$$
\pi: \begin{pmatrix} x\\ y\end{pmatrix} \mapsto \begin{pmatrix} \frac{x}{x+y} \\ \frac{y}{x+y} \end{pmatrix}\,,
$$
and the projective map is represented as
 $\pi\circ A:\Sigma \to\Sigma$,
taking the explicit form
\begin{equation*}
\pi\circ A:
\begin{pmatrix} x\\ 1-x \end{pmatrix} \mapsto
\begin{pmatrix} \frac{(a-b)x+b}{(a+c-b-d)x + b+d} \\ \frac{(c-d)x+d}{(a+c-b-d)x+b+d}\end{pmatrix}\,.
\end{equation*}
This projective mapping is completely determined by its first coordinate, thereby
motivating the following definition of
the self-map $T_A$ of the unit interval $X=[0,1]$:

\begin{defn}\label{taxasa}
For $A= \begin{pmatrix} a&b\\c&d\end{pmatrix} \in M_2^+(\mathbb{R}^+)$,
the \emph{induced map} $T_A:X\to X$ is defined by
$$
T_A(x)=\frac{(a-b)x+b}{(a+c-b-d)x+b+d}\,,
$$
the \emph{induced image} $X_A$ is defined by 
$$X_A=T_A(X)=\left[\frac{b}{b+d},\frac{a}{a+c}\right]\,,$$
and the \emph{induced inverse map} $S_A:X_A\to X$ is given by
$$
S_A(x)= T_A^{-1}(x)  = \frac{ (b+d)x - b}{-(a+c-b-d)x+a-b}
\,.
$$
\end{defn}

\begin{remark}
Defining $P=\begin{pmatrix} 1 & 0 \\ 1 & 1 \end{pmatrix}$,
the M\"obius maps $T_A$ and $S_A$ are represented, respectively, by the matrices 
$PAP^{-1}$ and $PA^{-1}P^{-1}$.
\end{remark}

\begin{remark}\label{invariantposmultiple}
The objects defined in Definition \ref{taxasa} do not change if the matrix $A$
is multiplied by a positive real number; that is,
if $t>0$, $A \in M_2^+(\mathbb{R}^+)$,
then $T_{tA}=T_A$ (hence $S_{tA}=S_A$), and $X_{tA}=X_A$.
\end{remark}

In view of (\ref{definequal2}) in the definition of $\mm$,
it suffices to restrict attention to matrices
of the following form:

\begin{notation}
Let $\na$ denote the set of matrices $A= \begin{pmatrix} a&b\\c&d\end{pmatrix} \in M_2^+(\mathbb{R}^+)$
such that $a+c\neq b+d$.
\end{notation}

\begin{lemma}
For $A= \begin{pmatrix} a&b\\c&d\end{pmatrix} \in M_2^+(\mathbb{R}^+)$, the map $T_A$ has a single fixed point $p_A=T_A(p_A)$ in $X$.
If $A \in \na$ then
\begin{equation}\label{paformula}
p_A 
= \frac{a-d-2b+\sqrt{(a-d)^2+4bc}}{2(a+c-b-d)}\,,
\end{equation}
and if $A\notin\na$ then
\begin{equation}\label{paformula2}
p_A = \frac{b}{2b+d-a}\,.
\end{equation}
\end{lemma}
\begin{proof}
Uniqueness follows from the fact that $A$ has all entries strictly positive, and the formulae (\ref{paformula})
and (\ref{paformula2})
are
straightforward computations.
\end{proof}

\subsection{Notation and matrix preliminaries}

For a matrix $A= \begin{pmatrix} a&b\\c&d\end{pmatrix} \in M_2^+(\mathbb{R}^+)$,
it will be useful to write
\begin{equation}\label{alphadef}
\alpha_A = a+c-b-d \,,
\end{equation}
\begin{equation}\label{betadef}
\beta_A = a-d-2b \,,
\end{equation}
\begin{equation}\label{gammadef}
\gamma_A = \sqrt{(a-d)^2+4bc}\,,
\end{equation}
noting that these quantities are related by the following identity:

\begin{lemma}\label{alphabetagammalemma}
For $A \in M_2^+(\mathbb{R}^+)$,
\begin{equation}\label{alphabetagamma}
 \gamma_A^2 - \beta_A^2 
=
4b \alpha_A \,.
\end{equation}
\end{lemma}
\begin{proof}
Straightforward computation.
\end{proof}

For ease of reference it will be convenient to collect together
 various previously defined objects expressed in terms of the above notation.

\begin{prop}
For $A \in M_2^+(\mathbb{R}^+)$, 
\begin{equation}\label{rhoaredone}
\rho_A = \frac{2b}{\beta_A + \gamma_A}\,,
\end{equation}
\begin{equation*}
T_A(x) = \frac{(a-b)x+b}{\alpha_Ax+b+d}\,,
\end{equation*}
\begin{equation*}
S_A(x)
= \frac{(b+d)x-b}{-\alpha_{A}(x+\sigma_{A})} \,,
\end{equation*}
and if moreover $A\in\na$ then
\begin{equation}\label{sigmaaredone}
\sigma_A = \frac{b-a}{\alpha_A}\,,
\end{equation}
\begin{equation}\label{paformularedone}
p_A = \frac{\beta_A+\gamma_A}{2\alpha_A} 
 = \frac{b\, \sigma_A}{(b-a)\rho_A}\,.
\end{equation}
The set $\mm$ can be written as
\begin{equation*}
\mm
=
\left\{ (A_0,A_1) \in \na^2 : \frac{a_0}{c_0}<\frac{b_1}{d_1}
\ \text{and}\ 
\alpha_{A_1} < 0 < \alpha_{A_0} \right\}\,.
\end{equation*}
\end{prop}

\subsection{Perron-Frobenius theory and the joint spectral radius}

\begin{lemma}\label{pflemma}
The dominant (Perron-Frobenius) eigenvalue $\lambda_A>0$ of the matrix
$A= \begin{pmatrix} a&b\\c&d\end{pmatrix} \in M_2^+(\mathbb{R}^+)$
is given by
\begin{equation*}
\label{lambdaadef}
\lambda_A = \frac{1}{2}\left( a+d +\gamma_A\right)
=\frac{b}{p_A}+a-b
\,,
\end{equation*}
with corresponding left eigenvector 
\begin{equation*}
w_A = ( a-d +\gamma_A, 2b)\,,
\end{equation*}
and right eigenvector
\begin{equation*}
v_A= \begin{pmatrix} p_A \\ 1-p_A \end{pmatrix} \,.
\end{equation*}
\end{lemma}

For $A \in M_2^+(\mathbb{R}^+)$,
the derivative of $T_A$ at its fixed point $p_A$ is related to the determinant
and Perron-Frobenius eigenvalue of $A$ as follows:

\begin{lemma}\label{derivfixedpt}
If $A \in M_2^+(\mathbb{R}^+)$ then
\begin{equation*}
T_A'(p_A) = \frac{\det A}{\lambda_A^2}\,.
\end{equation*}
\end{lemma}
\begin{proof}
This is a straightforward computation. If $A\in\na$ we can use the 
expression 
$p_A=\frac{\beta_A+\gamma_A}{2\alpha_A}$
(see (\ref{paformularedone})), the
derivative formula
$T_A'(x)
=
\det A\left( \alpha_A x+b+d\right)^{-2}
$, and the fact that
$\lambda_A = \frac{1}{2}\left( a+d +\gamma_A\right)
=\frac{1}{2}(\beta_A+\gamma_A)+b+d$
(see (\ref{lambdaadef})).
If $A\notin\na$ then $T_A'\equiv \frac{a-b}{b+d}$,
$\lambda_A=b+d$, and the relation $a+c=b+d$ means that $\det A = (a-b)(b+d)$, so the result follows.
\end{proof}

Since the Perron-Frobenius eigenvalue $\lambda_A$ is also the spectral radius $r(A)$,
we obtain the following corollary:

\begin{cor}\label{specradiusta}
If $A \in M_2^+(\mathbb{R}^+)$ then its spectral radius $r(A)$ satisfies
\begin{equation}\label{specradta}
r(A) = \left( \frac{\det A}{T_A'(p_A)}\right)^{1/2}\,.
\end{equation}
\end{cor}
\begin{proof}
Immediate from Lemma \ref{derivfixedpt}.
\end{proof}

\begin{notation}
Let us write finite words using the alphabet $\{0,1\}$ as
$\underline{i}=(i_1,\ldots,i_n)$, and their length as $|\underline{i}|=n$.
Let $\Omega^*$ denote the set of all such finite words; that is, $\Omega^*=\cup_{n\ge1} \{0,1\}^n$.
Given $\A=(A_0,A_1)\in M_2(\mathbb{R})^2$, and $\underline{i}\in\{0,1\}^n$, let $A(\underline{i})$
denote the product
\begin{equation}\label{matrixproduct}
A(\underline{i}) = A_{i_1}\cdots A_{i_n}\,.
\end{equation}
\end{notation}

Corollary \ref{specradiusta} then allows us to express
the joint spectral radius of a matrix pair
 $\A=(A_0,A_1)\in M_2^+(\mathbb{R}^+)^2$
in terms of induced maps of the products $A(\underline{i})$ as follows:

\begin{prop}\label{jsrderiv}
If $\A=(A_0,A_1)\in M_2^+(\mathbb{R}^+)^2$, then its joint spectral radius $r(\A)$ satisfies
\begin{equation}\label{jsrderiveq}
r(\A)=
\sup_{\underline{i}\in \Omega^*} 
\left( \frac{\det A(\underline{i}) }{T_{A(\underline{i}) }' (p_{A(\underline{i} )})} \right)^{1/2|\underline{i}|}\,.
\end{equation}
\end{prop}
\begin{proof}
The
expression (\ref{rjsr}) for the joint spectral radius can be written as
$$r(\A)
=\sup\left\{ r(A_{i_1}\cdots A_{i_n})^{1/n}:n\ge 1, i_j\in\{0,1\}\right\}
=
\sup_{\underline{i}\in \Omega^*} r(A(\underline{i}))^{1/|\underline{i}|}
\,,
$$
so applying Corollary \ref{specradiusta} with $A$ replaced by $A(\underline{i})$ yields the result.
\end{proof}

\subsection{Some useful formulae}
The purpose of this short subsection is to collect together various formulae
which will prove useful in the sequel.
Firstly, we have the following two expressions for the determinant of $A$ involving $\alpha_A$ and 
$\sigma_A$:

\begin{lemma}\label{detalphasigma}
For $A\in \na$, its determinant can be expressed as
\begin{equation}\label{detac}
\det A = -\alpha_A(a+c)\left( \frac{a}{a+c} + \sigma_A\right)
\end{equation}
and
\begin{equation}\label{detbd}
\det A = -\alpha_A(b+d)\left( \frac{b}{b+d} + \sigma_A\right) \,.
\end{equation}
\end{lemma}
\begin{proof}
Straightforward computation.
\end{proof}

There is a useful alternative way of expressing the quantity $\rho_A$:

\begin{lemma}\label{rhoquadratic}
For $A\in \na$,
\begin{equation}\label{altrho}
\rho_A = \frac{\gamma_A -\beta_A}{2\alpha_A}\,,
\end{equation}
and $\rho_A$
is the larger
 root of the quadratic polynomial $q_A$ defined by
\begin{equation}\label{qdef}
q_A(z) = \alpha_A z^2 +\beta_A z-b
\,.
\end{equation}
\end{lemma}
\begin{proof}
The expression (\ref{altrho}) follows from
(\ref{rhoaredone}) and the identity 
(\ref{alphabetagamma}).

The larger root of $q_A$ is computed to be
$$
\frac{-\beta_A+\sqrt{\beta_A^2+4\alpha_A b}}{2\alpha_A}
= \frac{-\beta_A +\gamma_A}{2\alpha_A}\,,
$$
again using (\ref{alphabetagamma}).
\end{proof}

Clearly
\begin{equation}\label{qdet}
q_A(z)
=\det \begin{pmatrix} 1 & z \\  -\alpha_A z & \beta_A z - b  \end{pmatrix} \,,
\end{equation}
though the following expression will prove to be more useful:

\begin{lemma}\label{qpolylemma}
For $A\in \na$,
\begin{equation*}
q_A(z) 
=\det \begin{pmatrix} 1 & z \\  b +d -\alpha_A z & (a-b) z -b  \end{pmatrix} \,.
\end{equation*}
\end{lemma}
\begin{proof}
Straightforward computation.
\end{proof}

\section{Projective convexity and projective concavity}\label{projconvprojconc}

\begin{remark}\label{derivativesremark}
\item[\, (a)]
For $A\in\na$ and $x\in X$, the derivative formula
\begin{equation}\label{taderiv}
T_A'(x)
=
\det A\left( \alpha_A x+b+d\right)^{-2}
\end{equation}
implies that if $\A\in \na^2$ then 
$T_{A_0}$ and $T_{A_1}$ are orientation preserving.

\item[\, (b)]
For $A\in\na$ and $x\in X$, 
the second derivative formula 
\begin{equation}\label{tasecondderiv}
T_A''(x)
=
-2\alpha_A \det A \left( \alpha_A x+b+d\right)^{-3}
\end{equation}
implies that if $\A\in\mm$ then 
$T_{A_0}'' < 0$ and $T_{A_1}'' > 0$,
i.e.~$T_{A_0}$ is strictly concave and $T_{A_1}$ is strictly convex.
\end{remark}

Part (b) of Remark \ref{derivativesremark} motivates the following definition, partitioning $\na$ into two subsets:

\begin{defn}
A matrix $A\in \na$ will be called \emph{projectively convex} if the induced map $T_A$ is strictly convex, and
\emph{projectively concave} if the induced map $T_A$ is strictly concave.
\end{defn}

\begin{remark}
The set $\na$ is the disjoint union of the subset of
projectively convex matrices and
the subset of projectively concave matrices.
\end{remark}

Recall that 
\begin{equation}\label{weig}
w_A=(w^{(1)}_A,w^{(2)}_A)
= ( a-d +\gamma_A, 2b)
\end{equation}
denotes the Perron-Frobenius left eigenvector of $A\in \na$,
and that (consequently) the right eigenvector for the other eigenvalue of $A$
is $\begin{pmatrix} w^{(2)}_A \\ -w^{(1)}_A\end{pmatrix}$.
It is useful to record the following identity:

\begin{lemma}\label{wrholem}
For $A = \begin{pmatrix} a&b\\c&d\end{pmatrix} \in \na$, 
\begin{equation}\label{wrho}
\rho_A = \frac{w_A^{(2)}}{w_A^{(1)}-w_A^{(2)}}\,.
\end{equation}
\end{lemma}
\begin{proof}
Immediate from
(\ref{rhoaredone}) and
 (\ref{weig}).
\end{proof}

\begin{cor}\label{rhosiminvt}
For $A\in \na$, if $Q\in M_2(\mathbb{R})$ is non-singular then
$\rho_{Q^{-1}AQ}=\rho_A$; that is, $\rho_A$ is invariant  under similarities.
\end{cor}
\begin{proof}
Immediate from Lemma \ref{wrholem}, and the fact that the eigenvector $w_A$ is invariant under similarities.
\end{proof}

There are various useful characterisations of projective convexity and projective concavity:

\begin{lemma}\label{concaveequivconditions}
For $A  \in \na$, the following are equivalent
\item[\, (i)] 
$A$ is projectively concave,
\item[\, (ii)] 
$\alpha_A>0$,
\item[\, (iii)]
$\rho_A>0$,
\item[\, (iv)]
$w^{(1)}_A > w^{(2)}_A$.
\end{lemma}
\begin{proof}
As noted in Remark \ref{derivativesremark} (b), the second derivative formula 
(\ref{tasecondderiv}) yields the equivalence of (i) and (ii), since $\det A >0$, and a function
is strictly concave if and only if its second derivative is strictly negative.

To prove the equivalence of (ii) and (iii), we consider separately the cases where $\beta_A\ge0$ and 
$\beta_A<0$.
If $\beta_A\ge0$ then $\alpha_A=\beta_A+b+c>0$, so we must simply
show that $\rho_A>0$.
But $\gamma_A>0$ by definition, hence $\beta_A+\gamma_A>0$, and therefore
(\ref{rhoaredone}) implies that $\rho_A = \frac{2b}{\beta_A + \gamma_A}>0$, as required.
If on the other hand $\beta_A<0$ then $\gamma_A - \beta_A>0$ is automatically true, 
again since $\gamma_A>0$ by definition.
Using (\ref{alphabetagamma}) and (\ref{rhoaredone}) we see that
\begin{equation*}
2\alpha_A \rho_A = \gamma_A-\beta_A >0\,,
\end{equation*}
so indeed $\alpha_A>0$ if and only if $\rho_A>0$, as required.

Lastly, the equivalence of (iii) and (iv) is immediate from (\ref{wrho}), since $w_A^{(2)}>0$.
\end{proof}

\begin{lemma}\label{convexequivconditions}
For $A  \in \na$, the following are equivalent
\item[\, (i)] 
$A$ is projectively convex,
\item[\, (ii)] 
$\alpha_A<0$,
\item[\, (iii)]
$\rho_A < -1$,
\item[\, (iv)]
$w^{(1)}_A < w^{(2)}_A$.
\end{lemma}
\begin{proof}
A function
is strictly convex if and only if its second derivative is strictly positive, so
the equivalence of (i) and (ii)
follows from
(\ref{tasecondderiv}), since $\det A >0$
and $\alpha_Ax+b+d = a+c+(b+d)(1-x)>0$ for all $x\in X$.

To prove that (iii) is equivalent to (iv), note that (\ref{wrho})
gives $w^{(1)}_A=w^{(2)}_A(1+\rho_A^{-1})$; therefore $\rho_A<-1$ if and only if
$1+\rho_A^{-1}\in (0,1)$, if and only if $w^{(1)}_A\in(0,w^{(2)}_A)$.

Lastly, to prove the equivalence of (ii) and (iii), it follows from
Lemma \ref{concaveequivconditions} that $\alpha_A<0$ if and only if $\rho_A<0$,
but this latter inequality in fact implies $w^{(2)}_A-w^{(1)}_A>0$ by (\ref{wrho}), so
 $$\rho_A = -1 - \frac{w_A^{(1)}}{w_A^{(2)}-w_A^{(1)}} <-1\,,$$ 
 as required.
\end{proof}

Note that in Lemma \ref{convexequivconditions} the assertion is not merely that $\rho_A<0$, but that $\rho_A<-1$; this should be contrasted with the inequality $\rho_A>0$ in Lemma
\ref{concaveequivconditions}.

It is now clear why $\mm$ is described as the set of 
\emph{concave-convex pairs}\footnote{Note, however, the restriction that the induced images be 
\emph{disjoint}, with the concave image to the \emph{left} of the convex one.}:

\begin{prop}\label{concavconvexprop}
The set $\mm$ consists of those matrix pairs  $(A_0,A_1)\in \na^2$ such that
$A_0$ is projectively concave, $A_1$ is projectively convex, and the induced
image for $A_0$ is strictly to the left of the induced image of $A_1$. 
\end{prop}
\begin{proof}
Lemmas \ref{concaveequivconditions} and \ref{convexequivconditions}
imply that the inequality $\alpha_{A_1}<0<\alpha_{A_0}$
in Definition \ref{frakmdefn}
 is equivalent to
$A_0$ being projectively concave and $A_1$ being projectively convex.
The inequality $\frac{a_0}{c_0}  <  \frac{b_1}{d_1}$ in Definition \ref{frakmdefn}
is equivalent to
$T_{A_0}(1) = \frac{a_0}{a_0+c_0}  <  \frac{b_1}{b_1+d_1}=T_{A_1}(0)$, which asserts that the right
endpoint of the induced image $X_{A_0}$ is  strictly to the left of the left endpoint
of the induced image $X_{A_1}$.
\end{proof}

\begin{lemma}\label{lemma1}
If $A\in \na$ is projectively concave then
$\sigma_{A}<0$. 
\end{lemma}
\begin{proof}
Projective concavity of $A$ means that $\alpha_A>0$,
 so
 by (\ref{sigmaaredone})
  it suffices to show that $b<a$.
Since $\det A = ad-bc>0$ 
and $\alpha_A=a+c-b-d>0$
we derive
$$
a-b>d-c > \frac{bc}{a}-c = -\frac{c}{a}(a-b)\,,
$$
or in other words
$$
(a-b)\left(1+\frac{c}{a}\right)>0\,,
$$
and hence $a-b>0$, 
as required.
\end{proof}

We can now prove the following result mentioned in \S \ref{statementsubsection}
(note, however, that there is no constraint on the sign of $\sigma_{A_1}$ when $(A_0,A_1)\in\mm$):

\begin{cor}\label{rhoposlessminusone}
If $(A_0,A_1)\in \mm$ then 
$\sigma_{A_0}<0<\rho_{A_0}$ and $\rho_{A_1}<-1$.
\end{cor}
\begin{proof}
Immediate from Lemmas \ref{concaveequivconditions},  \ref{convexequivconditions},
and \ref{lemma1}
\end{proof}

An important result is the following:

\begin{lemma}\label{posneglinfactors}
If $A\in \na$ then
\begin{equation}\label{alwayspos}
-\alpha_{A}(x+\sigma_{A}) > 0 \quad\text{for all } x\in X_{A}\,.
\end{equation}

In particular, if $A\in \na$ is projectively concave then
\begin{equation}\label{posneglinfactorseq0}
x+\sigma_{A}<0 \quad\text{for all }x\in X_{A}\,,
\end{equation}
and if $A\in \na$ is projectively convex then
\begin{equation}\label{posneglinfactorseq1}
x+\sigma_{A}>0 \quad\text{for all }x\in X_{A}\,.
\end{equation}
\end{lemma}
\begin{proof}
Clearly (\ref{alwayspos}) follows from (\ref{posneglinfactorseq0}) and
(\ref{posneglinfactorseq1}),
since $\alpha_A$ is positive if $A$ is projectively concave, and negative if
$A$ is projectively convex,
 by Lemmas 
\ref{concaveequivconditions} and \ref{convexequivconditions}.

To prove (\ref{posneglinfactorseq0}), note that 
$ -\alpha_A(a+c)\left( \frac{a}{a+c} + \sigma_A\right) = \det A >0$
by (\ref{detac}), and if $A$ is projectively concave then $\alpha_A>0$, so
$\frac{a}{a+c} + \sigma_A < 0$. But $\frac{a}{a+c}$ is the righthand endpoint of $X_A$, so if $x\in X_A$ then $x\le \frac{a}{a+c}$, therefore
$
x+\sigma_A \le \frac{a}{a+c} + \sigma_A < 0
$, as required.

To prove (\ref{posneglinfactorseq1}), note that 
$-\alpha_A(b+d)\left( \frac{b}{b+d} + \sigma_A\right) = \det A >0$
by (\ref{detbd}), and if $A$ is projectively convex then $\alpha_A<0$, so
$\frac{b}{b+d} + \sigma_A >0$.
But $\frac{b}{b+d}$ is the lefthand endpoint of $X_A$, so if $x\in X_A$ then $x\ge \frac{b}{b+d}$, therefore
$
x+\sigma_A \ge \frac{b}{b+d} + \sigma_A > 0
$, as required.
\end{proof}

\begin{cor}\label{alwaysposcor}
If $(A_0,A_1)\in\mm$ then 
$
x+\sigma_{A_0} <0
$
for $x\in X_{A_0}$,
and 
$x+\sigma_{A_1} >0
$
for $x\in X_{A_1}$, and
$-\alpha_{A_i}(x+\sigma_{A_i}) > 0$ for all $x\in X_{A_i}$, $i\in\{0,1\}$.
\end{cor}
\begin{proof}
Immediate from Lemma \ref{posneglinfactors}.
\end{proof}

\begin{lemma}\label{qrhoineq}
If $\A=(A_0,A_1)\in\mm$ then 
\begin{equation*}
q_{A_1}(\rho_{A_0}) < 0
<
q_{A_0}(\rho_{A_1})\,.
\end{equation*}
\end{lemma}
\begin{proof}
The larger root of $q_{A_1}$ is $\rho_{A_1}$, by Lemma \ref{rhoquadratic}.
It follows that $q_{A_1}(z)= \alpha_{A_1} z^2 +\beta_{A_1} z-b_1 <0$ for all $z>\rho_{A_1}$, since 
the leading coefficient $\alpha_{A_1}<0$, since $A_1$ is projectively convex.
But by Lemmas \ref{concaveequivconditions} and \ref{convexequivconditions}
 we know that $\rho_{A_0} > 0 > -1 > \rho_{A_1}$, so indeed
$q_{A_1}(\rho_{A_0}) < 0$, as required.

The smaller root of $q_{A_0}$, which we shall denote by $r_{A_0}$, is given by
$$
r_{A_0} = \frac{-(\gamma_{A_0}+\beta_{A_0})}{2\alpha_{A_0}}\,.
$$
It follows that 
\begin{equation}\label{qlarger}
q_{A_0}(z)= \alpha_{A_0} z^2 +\beta_{A_0} z-b_0 >0
\quad\text{for all }
z<r_{A_0}\,,
\end{equation} 
since 
the leading coefficient $\alpha_{A_0}>0$, since $A_0$ is projectively concave.
Now $\rho_{A_1}<-1$ by Lemma \ref{convexequivconditions}, and if we can show that $r_{A_0}> -1$
then it follows that $\rho_{A_1}<r_{A_0}$, and hence $q_{A_0}(\rho_{A_1})>0$ by (\ref{qlarger}).

To show that indeed $r_{A_0}> -1$, note that this inequality is equivalent to $2\alpha_A-\beta_A>\gamma_A$.
Both sides are positive, so this is equivalent to
$(2\alpha_A-\beta_A)^2>\gamma_A^2$,
which
using (\ref{alphabetagamma})
 becomes $4\alpha_A(\alpha_A-\beta_A) > 4b\alpha_A$.
 This latter inequality is equivalent to $\alpha_A-\beta_A>b$, which is true because in fact
 $\alpha_A-\beta_A=b+c>b$.
\end{proof}

We deduce the following technical lemma, which will be used in \S \ref{dominatessec}:

\begin{lemma}\label{techderivlemma}
If $\A\in\mm$ then the M\"obius function
\begin{equation*}
x\mapsto \frac{x+\rho_{A_0}}{(b_1+d_1-\alpha_{A_1}\rho_{A_0})x +(a_1-b_1)\rho_{A_0}-b_1}
\end{equation*}
has strictly negative derivative,
while
the M\"obius function
\begin{equation*}
x\mapsto \frac{x+\rho_{A_1}}{(b_0+d_0-\alpha_{A_0}\rho_{A_1})x +(a_0-b_0)\rho_{A_1}-b_0}
\end{equation*}
has strictly positive derivative.
\end{lemma}
\begin{proof}
A general M\"obius map $x\mapsto \frac{Px+Q}{Rx+S}$ has derivative
$D(Rx+S)^{-2}$, where 
$D=PS-QR$,
so the derivative is strictly negative if $D<0$, and strictly positive if $D>0$.
For our first M\"obius map we have
$$
D=\det \begin{pmatrix} 1 & \rho_{A_0} \\  b_1 +d_1 -\alpha_{A_1} \rho_{A_0} & (a_1-b_1) \rho_{A_0} -b_1  \end{pmatrix}
= q_{A_1}(\rho_{A_0})
$$
by Lemma \ref{qpolylemma},
and $q_{A_1}(\rho_{A_0})$ is strictly negative by Lemma \ref{qrhoineq}, so the derivative of the map is strictly negative, as required.

For our second M\"obius map we have
$$
D=\det \begin{pmatrix} 1 & \rho_{A_1} \\  b_0 +d_0 -\alpha_{A_0} \rho_{A_1} & (a_0-b_0) \rho_{A_1} -b_0  \end{pmatrix}
= q_{A_0}(\rho_{A_1})
$$
by Lemma \ref{qpolylemma},
and $q_{A_0}(\rho_{A_1})$ is strictly positive by Lemma \ref{qrhoineq}, so the derivative of the map is strictly positive, as required.
\end{proof}

\section{The induced dynamical system for a concave-convex matrix pair}\label{induceddynsyssection}

\subsection{The induced dynamical system and joint spectral radius}

\begin{defn}\label{induceddefn}
For a matrix pair $\A=(A_0,A_1)\in\mm$,
define the
\emph{induced space} $X_\A$ to be
$$
X_\A = X_{A_0}\cup X_{A_1}\,,
$$
and define the \emph{induced dynamical system}
$T_\A:X_\A\to X$ by
\begin{equation*}
T_\A(x) = 
\begin{cases}
S_{A_0}(x)&\text{if } x\in X_{A_0}\\
S_{A_1}(x)&\text{if }x\in X_{A_1}\,.
\end{cases}
\end{equation*}
\end{defn}

\begin{remark}\label{taremarks}
\item[\, (a)]
The map $T_\A:X_\A\to X$ is Lipschitz continuous since $X_\A=X_{A_0}\cup X_{A_1}$ is the union of disjoint intervals $X_{A_0}$ and $X_{A_1}$, and the restriction of $T_\A$ to $X_{A_i}$ is the M\"obius mapping $S_{A_i}$, which is certainly Lipschitz continuous.
\item[\, (b)]
Note that the (surjective) induced dynamical system $T_\A$ is naturally defined as a mapping
from $X_\A$ to $X=[0,1]$.
To view it as a surjective \emph{self}-mapping of some set (the natural setting for a dynamical system) we consider its restriction to 
the \emph{induced Cantor set}
$Y_\A:=\cap_{n\ge0}T_\A^{-n}(X_\A)$,
and note that $T_\A:Y_\A\to Y_\A$ is topologically conjugate to the shift map 
$\sigma$ on $\Omega = \{0,1\}^\N$.
\end{remark}

\begin{prop}\label{jsrderivta}
If $\A=(A_0,A_1)\in\mm$ then its joint spectral radius $r(\A)$ satisfies
\begin{equation}\label{jsrderivtaeq}
r(\A)=
\sup_{\underline{i}\in \Omega^*} 
\left( \det A(\underline{i}) \,\, (T_{\A}^{|\underline{i}|})' (p_{A(\underline{i} )}) \right)^{1/2|\underline{i}|}\,.
\end{equation}
\end{prop}
\begin{proof}
From Definition \ref{induceddefn} we see that $T_\A\circ T_{A_i}$ is the identity map on $X$, 
for $i\in\{0,1\}$, since $T_\A$ is defined in terms of the inverses $S_{A_i}=T_{A_i}^{-1}$. 
Similarly, for any $\underline{i}\in\{0,1\}^n$ we see that $T_\A^n\circ T_{A(\underline{i})}$ is also the
identity map on $X$, so
$(T_\A^n)'(T_{A(\underline{i})}(x)) \, T_{A(\underline{i})}'(x) =1$
for all $x\in X$, by the chain rule. Setting $x=p_{A(\underline{i})} = T_{A(\underline{i})}(p_{A(\underline{i})})$ we obtain
\begin{equation}\label{inversederivformu}
(T_\A^n)'(p_{A(\underline{i})}) = \frac{1}{ T_{A(\underline{i})}'( p_{A(\underline{i})} )}  \,,
\end{equation}
and combining this with (\ref{jsrderiveq}) gives the required formula (\ref{jsrderivtaeq}).
\end{proof}

\subsection{Invariant measures for the induced dynamical system}

\begin{defn}
For $\A\in\mm$,
let $\m_\A$ denote the set
of $T_\A$-invariant Borel probability measures on $X=[0,1]$;
the support of any such measure is contained in $Y_\A=\cap_{n\ge0} T_\A^{-n}(X_\A)$.
\end{defn}

The following is a well known consequence of the compactness of $Y_\A$ and continuity of $T_\A$
(see e.g.~\cite[Thm.~6.10]{walters}):

\begin{lemma}\label{macompactlemma}
The set $\m_\A$ is compact with respect to the weak$^*$ topology.
\end{lemma}

\begin{defn}\label{periodicorbitmeasuredefn}
For any $p\in X_\A$ that is a periodic point for $T_\A$, with $T_\A^n(p)=p$, we say that the 
probability measure $\mu\in\m_\A$ defined by
\begin{equation}
\mu
=\frac{1}{n}\sum_{j=0}^{n-1} \delta_{T_\A^j (p)}
\end{equation}
is the corresponding \emph{periodic orbit measure} (or \emph{$T_\A$-periodic orbit measure}).
\end{defn}

\begin{remark}\label{conjugacymeasures}
The topological conjugacy  $h_\A:\Omega\to Y_\A$ between the shift map $\sigma:\Omega\to\Omega$
and $T_\A:Y_\A\to Y_\A$ (cf.~Remark \ref{taremarks} (b))
induces a one-to-one correspondence 
$h_\A^*:\m\to\m_\A$ between invariant measures.
\end{remark}

\begin{defn}\label{eomax}
For a bounded Borel function $f:X_\A\to\mathbb{R}$,
a measure $m\in\m_\A$ is called \emph{$f$-maximizing} if
$$
\int f\, dm = \sup_{\mu\in\m_\A} \int f\, d\mu 
\,.
$$
\end{defn}

In the generality of Definition \ref{eomax}, the notion of  an $f$-maximizing invariant measure
is part of the wider field of
so-called ergodic optimization, see e.g.~\cite{jeo}.

\begin{defn}
For $\A=(A_0,A_1)\in\mm$, define the \emph{induced function} $f_\A:X_\A\to\mathbb{R}$ by
\begin{equation}\label{function}
f _\A =  \frac{1}{2}\left( \log T_\A' +(\log \det A_0)\mathbbm{1}_{X_{A_0}}   +(\log \det A_1)\mathbbm{1}_{X_{A_1}}  \right)  \,. \
\end{equation}
That is,
\begin{equation}\label{fthatis}
f_\A(x)=
\begin{cases}
\frac{1}{2}\left( \log S_{A_0}'(x) + \log \det A_0  \right) &\text{if }x\in X_{A_0}   \\
\frac{1}{2} \left(\log S_{A_1}'(x) + \log \det A_1 \right) &\text{if }x\in X_{A_1}   \,,
\end{cases}
\end{equation}
so writing $A_i= \begin{pmatrix} a_i&b_i\\c_i&d_i\end{pmatrix} $ gives
\begin{equation}\label{explicitfalt}
f_\A(x) =  \log  \left( \frac{\det A_i}{-\alpha_{A_i}(x+ \sigma_{A_i})} \right)
\quad\text{for }x\in X_{A_i}\,,
\end{equation}
where we recall from
Corollary \ref{alwaysposcor} that $\frac{\det A_i}{-\alpha_{A_i}(x+ \sigma_{A_i})}>0$
for all $x\in X_{A_i}$.
\end{defn}

\begin{remark}\label{faremark}
The function $f_\A$ is clearly Lipschitz continuous on each $X_{A_i}$, hence Lipschitz continuous
on $X_\A = X_{A_0}\cup X_{A_1}$, since the intervals $X_{A_0}$ and $X_{A_1}$ are disjoint.
\end{remark}

The reason for introducing the function $f_\A$ is provided by the following characterisation of the joint spectral radius in terms of ergodic optimization:

\begin{theorem}\label{maxtheorem}
If $\A=(A_0,A_1)\in\mm$ then its joint spectral radius $r(\A)$ satisfies
\begin{equation}\label{logra}
\log r(\A) = 
\max_{\mu\in\m_\A} \int f_\A\, d\mu  \,.
\end{equation}
\end{theorem}  
\begin{proof}
From Proposition \ref{jsrderivta} we have
\begin{equation}\label{jsrderivtalog} 
\log r(\A)=
\sup_{\underline{i}\in \Omega^*} 
\log \left( \det A(\underline{i}) \,\, (T_{\A}^{|\underline{i}|})' (p_{A(\underline{i} )}) \right)^{1/2|\underline{i}|}\,.
\end{equation}

If $\underline{i}\in\{0,1\}^n$ then
\begin{align}\label{falogdet}
\log \left( \det A(\underline{i}) \,\, (T_{\A}^n)' (p_{A(\underline{i} )}) \right)^{1/2|\underline{i}|}
&=
\frac{1}{2n} \left( \log \det \left( A_{i_1}\cdots  A_{i_n} \right) + \log (T_{\A}^n)' (p_{A(\underline{i} )}) \right) \cr
&=
\frac{1}{2n} \left( \log \prod_{j=1}^n \det A_{i_j}  + \log  \prod_{j=0}^{n-1}  T_{\A}' (T_\A^j (p_{A(\underline{i} )})) \right) \cr
&=
\frac{1}{n} \sum_{j=0}^{n-1}\  \frac{1}{2}\left( \log \det A_{i_{j+1}}  + \log   T_{\A}' (T_\A^j (p_{A(\underline{i} )})) \right) \cr
&=
\frac{1}{n} \sum_{j=0}^{n-1} f_\A(T_\A^j (p_{A(\underline{i} )})) \,,
\end{align}
where the last step uses (\ref{function}) together with the fact that
$$
\log \det A_{i_{j+1}}
=
\left( \left( \log \det A_0 \right) \mathbbm{1}_{X_{A_0}}   +(\log \det A_1)\mathbbm{1}_{X_{A_1}} \right)(T_\A^j (p_{A(\underline{i} )})) \,,
$$
because
\begin{equation*}
\mathbbm{1}_{X_{A_k}}(T_\A^j (p_{A(\underline{i} )}))
=
\begin{cases}
1& \text{if } i_{j+1}=k\cr
0& \text{if } i_{j+1}\neq k \,.
\end{cases}
\end{equation*}
Combining (\ref{jsrderivtalog}) and (\ref{falogdet}) gives
\begin{equation}\label{perptmeasures}
\log r(\A)=
\sup_{\underline{i}\in \Omega^*} 
\frac{1}{|\underline{i}|} \sum_{j=0}^{|\underline{i}|-1} f_\A(T_\A^j (p_{A(\underline{i} )})) 
= \sup_{\underline{i}\in \Omega^*} 
\int f_\A\, d\mu_{\underline{i}}
\,,
\end{equation}
where 
$$\mu_{\underline{i}}=\frac{1}{|\underline{i}|}\sum_{j=0}^{|\underline{i}|-1} \delta_{T_\A^j (p_{A(\underline{i} )})}$$
is the 
periodic orbit measure (see Definition \ref{periodicorbitmeasuredefn})
for the period-$|\underline{i}|$ point 
$p_{A(\underline{i} )}=T_\A^{|\underline{i}|}(p_{A(\underline{i} )})$, i.e.~the
unique measure in $\m_\A$ whose support equals the
periodic orbit $\{ T_\A^j (p_{A(\underline{i} )})\}_{j= 0}^{ |\underline{i}|-1}$.
By a result of Parthasarathy \cite{parthasarathy} (see also Sigmund \cite{sigmund}),
the collection of periodic orbit measures $\{\mu_{\underline{i}}:\underline{i}\in\Omega^*\}$
is weak$^*$ dense in the weak$^*$ compact space $\m_\A$, so 
\begin{equation}\label{parthsigmund}
\sup_{\underline{i}\in \Omega^*} 
\int f_\A\, d\mu_{\underline{i}}
= \max_{\mu\in\m_\A} \int f_\A\, d\mu \,,
\end{equation}
and combining with (\ref{perptmeasures}) gives the required equality
(\ref{logra}).
\end{proof}

\subsection{The finiteness property and periodic orbits}

In view of Theorem \ref{maxtheorem}, we shall be interested in those measures $m\in\m_\A$ which
are $f_\A$-maximizing, in the sense of Definition \ref{eomax},
 i.e.~$m$ attains the maximum in (\ref{logra}): $\log r(\A) = 
\max_{\mu\in\m_\A} \int f_\A\, d\mu  =\int f_\A\, dm$.
The finiteness property for $\A$ 
(which we recall means that $r(\A)=r(A_{i_1}\cdots A_{i_n})^{1/n}$
for some $i_1,\ldots, i_n\in\{0,1\}$)
corresponds to existence of a periodic orbit measure which
is $f_\A$-maximizing:

\begin{prop}\label{finitenessperiodicequivalenceprop}
$\A\in\mm$ has the finiteness property if and only if some $T_\A$-periodic orbit measure is $f_\A$-maximizing.
\end{prop}
\begin{proof}
If $\A\in\mm$ has the finiteness property, 
and $\underline{i}\in\{0,1\}^n$ satisfies $r(\A)= r(A(\underline{i}))^{1/n}$,
then we claim that the corresponding periodic orbit measure
$\mu_{\underline{i}}=\frac{1}{n}\sum_{j=0}^{n-1} \delta_{T_\A^j (p_{A(\underline{i} )})}$
is $f_\A$-maximizing.
To see this, first note that (\ref{logra}) gives
\begin{equation}\label{firstnote}
\log r(A(\underline{i}))^{1/n} = \log r(\A) = \max_{\mu\in\m_\A} \int f_\A\, d\mu\,,
\end{equation}
and the lefthand side of (\ref{firstnote}) can be written as
\begin{equation}\label{intermediate1}
\log r(A(\underline{i}))^{1/n} = \log \left( \frac{\det A(\underline{i})}{T_{A(\underline{i})}'(p_{A(\underline{i})})}\right)^{1/2n}
= \log \left( \det A(\underline{i}) \,\, (T_{\A}^n)' (p_{A(\underline{i} )}) \right)^{1/2n}
\end{equation}
using Corollary \ref{specradiusta} and (\ref{inversederivformu}),
and therefore (\ref{falogdet}) gives
\begin{equation}\label{intermediate2}
\log r(A(\underline{i}))^{1/n} = \frac{1}{n} \sum_{j=0}^{n-1} f_\A(T_\A^j (p_{A(\underline{i} )})) 
=\int f_\A\, d\mu_{\underline{i}}\,,
\end{equation}
so (\ref{firstnote}) implies that
$\mu_{\underline{i}}$
is indeed $f_\A$-maximizing.

Conversely, if some $T_\A$-periodic orbit measure is $f_\A$-maximizing,
then this measure is necessarily of the form $\mu_{\underline{i}}=\frac{1}{n}\sum_{j=0}^{n-1} \delta_{T_\A^j (p_{A(\underline{i} )})}$
for some $n\in\N$ and $\underline{i}\in\{0,1\}^n$,
and satisfies
$\int f_\A\, d\mu_{\underline{i}} = \max_{\mu\in\m_\A} \int f_\A\, d\mu$.
Combining
 (\ref{logra}) with
 (\ref{intermediate1}) and (\ref{intermediate2}) gives
 $\log r(A(\underline{i}))^{1/n} = \log r(\A)$, so $\A$ has the finiteness property, as required.
\end{proof}

Note that the above proof has also established:

\begin{prop}
Suppose that $\A\in\mm$,
and $\underline{i}\in\{0,1\}^n$ for some $n\in\N$.
Then $r(\A)= r(A(\underline{i}))^{1/n}$
if and only if the periodic orbit measure
$\mu_{\underline{i}}=\frac{1}{n}\sum_{j=0}^{n-1} \delta_{T_\A^j (p_{A(\underline{i} )})}$
is $f_\A$-maximizing.
\end{prop}

Recall that we say $\A\in\mm$ is a \emph{finiteness counterexample} if
 $r(\A) > r(A_{i_1}\cdots A_{i_n})^{1/n}$
for all $n\in\N$ and all choices $i_1,\ldots, i_n\in\{0,1\}$. We have:

\begin{prop}\label{counterexampleequivalenceprop}
$\A\in\mm$ is a finiteness counterexample if and only if no $T_\A$-periodic orbit measure is $f_\A$-maximizing.
In this case there exists at least one measure $\mu\in\m_\A$ that is $f_\A$-maximizing, and there exist uncountably many sequences $\omega\in\Omega$ such that
\begin{equation}
\label{inthelimit}
r(\A) = \lim_{n\to\infty} r(A_{\omega_1}\cdots A_{\omega_n})^{1/n} \,.
\end{equation}
\end{prop}
\begin{proof}
The first statement is equivalent to that of Proposition \ref{finitenessperiodicequivalenceprop}, while the existence of
an $f_\A$-maximizing measure $\mu$ is a consequence 
(see e.g.~\cite[Prop.~2.4 (i)]{jeo})
of the continuity of $f_\A$ and the weak$^*$ compactness of
$\m_\A$ (see Lemma \ref{macompactlemma}). In fact
$\mu$ may be chosen to be an ergodic measure,
since it is readily shown that the set of $f_\A$-maximizing measures is convex, and any of its extremal points 
is ergodic (see e.g.~\cite[Prop.~2.4]{jeo}).
The ergodic
theorem (see e.g.~\cite[Thm.~1.14]{walters}) then implies that
\begin{equation}\label{inthelimit2}
\log r(\A) = \int f_\A\, d\mu = \lim_{n\to\infty} \frac{1}{n}\sum_{j=0}^{n-1} f_\A(T_\A^j(p)) 
\end{equation}
for $\mu$-almost every $p\in Y_\A$.
Since no periodic orbit measure is $f_\A$-maximizing, the measure $\mu$ 
must have uncountable support, and therefore
(\ref{inthelimit2}) holds for an uncountable set of points $p\in Y_\A$.

We may use the topological conjugacy $h_\A:\Omega\to Y_\A$
to define the image measure $m=(h_\A^{-1})^*(\mu)$, which also has uncountable support, and
if we write $h_\A(\omega)=p$ then
\begin{equation}\label{sumproductexp}
r(A_{\omega_1}\cdots A_{\omega_n})^{1/n} = \exp\left( \frac{1}{n}\sum_{j=0}^{n-1} f_\A(T_\A^j(p)) \right) \,,
\end{equation} so
(\ref{inthelimit2}) implies $
r(\A) = \lim_{n\to\infty} r(A_{\omega_1}\cdots A_{\omega_n})^{1/n} 
$
for $m$-almost every $\omega\in\Omega$, hence for uncountably many $\omega\in\Omega$.
\end{proof}

\subsection{Monotonicity properties and formulae}

The following simple lemma records that for $\A\in\mm$, 
the induced dynamical system $T_{\A(t)}$
is independent of $t$, and that the induced function $f_{\A(t)}$ differs from $f_\A$
only by the addition of a scalar multiple of the characteristic function
for the image $X_{A_1}$.

\begin{lemma}\label{simple}
For $\A=(A_0,A_1)\in \mm$,
and all $t>0$,
\begin{itemize}
\item[(i)] $T_{\A(t)} = T_\A$ \,,
\item[(ii)] $f_{\A(t)} = f_\A + (\log t)\mathbbm{1}_{X_{A_1}}$,
\item[(iii)] 
$
f_{\A(t)}(T_{A_0}(1)) - f_{\A(t)}(T_{A_1}(0))
=
f_{\A}(T_{A_0}(1)) - f_{\A}(T_{A_1}(0)) - \log t\,,
$
\item[(iv)] $f_{\A(t)}'=f_\A'$, with 
\begin{equation}\label{fderivative}
f_\A' (x)= 
-( x+\sigma_{A_i})^{-1}
\quad\text{for }x\in X_{A_i}, \ i\in\{0,1\}\,.
\end{equation}
\end{itemize}
\end{lemma}
\begin{proof}
(i) From Remark \ref{invariantposmultiple} we see that if $t>0$ then
$T_{tA_1}=T_{A_1}$, hence
$T_{\A(t)} = T_\A$.

(ii) 
Formula (\ref{fthatis}) gives $f_\A = f_{\A(t)}$ on $X_{A_0}$,
while for $x\in X_{A_1}$ we have
$$f_{\A(t)}(x)  = \frac{1}{2}\left(  \log S_{A_1}'(x) + \log \det tA_1 \right)
= \log t +f_\A(x)$$
since $\log \det tA_1 = \log \left( t^2 \det A_1 \right) = 2\log t +\log \det A_1$,
thus
$f_{\A(t)} = f_\A + (\log t)\mathbbm{1}_{X_{A_1}}$.

(iii) This is immediate from part (ii).

(iv) The formula for $f_\A'$ follows readily from the explicit formula (\ref{explicitfalt}) for $f_\A$,
and is equal to $f_{\A(t)}'$ by (ii) above.
\end{proof}

\begin{lemma}\label{posnegf}
If $\A=(A_0,A_1)\in\mm$ then
\begin{enumerate}
\item[(i)] 
$f_\A'$ is strictly positive on $X_{A_0}$ and strictly negative on $X_{A_1}$, 
\item[(ii)]
$f_\A$ is strictly increasing on $X_{A_0}$ and strictly decreasing on $X_{A_1}$,
\item[(iii)]
$(f_\A\circ T_{A_0}^i)'(x)>0$ and $(f_\A\circ T_{A_1}^i)'(x)<0$ for all $i\ge1$, $x\in X$.
\end{enumerate}
\end{lemma}
\begin{proof}
(i)
In view of  formula (\ref{fderivative}), 
it suffices to note that by Corollary \ref{alwaysposcor},
$
x+\sigma_{A_0} <0
$
for $x\in X_{A_0}$,
and 
$x+\sigma_{A_1} >0
$
for $x\in X_{A_1}$.

(ii) This is an immediate consequence of (i).

(iii) 
By the chain rule,
\begin{equation}\label{chain}
(f_\A\circ T_{A_j}^i)'(x) = f_\A'(T_{A_j}^i (x)) (T_{A_j}^i)'(x)\,.
\end{equation}

The second factor $(T_{A_j}^i)'(x)$ on the righthand side of (\ref{chain}) is strictly positive
for all $x\in X$, $i\ge1$, $j\in\{0,1\}$, since $T_{A_j}$ is orientation-preserving, as noted in Remark \ref{derivativesremark}.

Regarding the sign of the first factor $f_\A'(T_{A_j}^i (x))$ on the righthand side of (\ref{chain}),
note that since $i\ge1$ then $T_{A_j}^i(x)\in X_{A_j}=T_{A_j}(X)$ for all $x\in X$.
Part (i) above then implies that $f_\A'(T_{A_j}^i (x))$
is strictly positive when $j=0$ and strictly negative when $j=1$.
It follows that
$(f_\A\circ T_{A_j}^i)'(x)$ 
is strictly positive when $j=0$ and strictly negative when $j=1$, as required.
\end{proof}

For the purposes of the following Lemma \ref{usefullem}, it will be convenient to introduce the following notation:

\begin{notation}
For a matrix $A= \begin{pmatrix} a&b\\c&d\end{pmatrix} \in \na$,
define
$$
\delta_A = \frac{b+d}{\alpha_A} = \frac{b+d}{a+c-b-d}\,.
$$
\end{notation}

We can now give another characterisation of $\rho_A$:

\begin{lemma}\label{usefullem}
For $A\in \na$,
\begin{equation*}
\rho_A = \lim_{k\to \infty} \delta_{A^k}\,.
\end{equation*}
\end{lemma}
\begin{proof}
Perron-Frobenius theory (see e.g.~\cite[Thm.~0.17]{walters}) gives
$$
\lim_{k\to\infty} \lambda_A^{-k} A^k = v \, w_A
=
\begin{pmatrix} v^{(1)} w_A^{(1)} & v^{(1)} w_A^{(2)}  \\  v^{(2)} w_A^{(1)} & v^{(2)} w_A^{(2)} \end{pmatrix}\,,
$$
where the positive dominant eigenvalue $\lambda_A>0$ and corresponding left eigenvector
$w_A$ are as in Lemma \ref{pflemma}, and $v$ is a corresponding right eigenvector
(a suitable multiple of $v_A$ from Lemma \ref{pflemma}), normalised so that
$w_A v=1$.

It follows that
\begin{equation}\label{rklimit}
\lim_{k\to\infty} \delta_{A^k}
= \frac{v^{(1)} w^{(2)}_A +v^{(2)} w_A^{(2)} }{v^{(1)} w_A^{(1)}+v^{(2)} w^{(1)}_A - v^{(1)} w^{(2)}_A -v^{(2)} w^{(2)}_A }
=\frac{w^{(2)}_A}{w^{(1)}_A -w^{(2)}_A}\,,
\end{equation}
so the formula
$\rho_A = \frac{w_A^{(2)}}{w_A^{(1)}-w_A^{(2)}}$ from Lemma \ref{wrholem} concludes the proof.
\end{proof}

\begin{cor}\label{useful}
For $A\in \na$, $x\in X$,
\begin{equation}\label{rhooccurence}
\sum_{n=1}^\infty (\log S_A'\circ T_A^n)'(x) = \frac{2}{x+\rho_A}\,.
\end{equation}
\end{cor}
\begin{proof}
A simple calculation using the chain rule 
yields
\begin{equation}\label{negativeformula}
\sum_{n=1}^k ( \log S_A'\circ T_A^n)'(x) =  
- (\log T_{A^k}')'(x)
= \frac{2}{x+\delta_{A^k}}
\end{equation}
for all $k\ge1$, so
letting $k\to\infty$ 
we see that the result follows from Lemma \ref{usefullem}.
\end{proof}

Recalling from (\ref{fthatis}) that  $f_\A= \frac{1}{2}\left( \log S_{A_i}'  +\log \det A_i\right)$ on $X_{A_i}$, 
the following result is an immediate consequence of Corollary \ref{useful}:

\begin{cor}\label{usefulcor}
If $\A=(A_0,A_1)\in \na^2$ then for $i\in \{0,1\}$,
\begin{equation}\label{sumformula}
\sum_{n=1}^\infty (f_\A\circ T_{A_i}^n)'(x) = \frac{1}{x+\varrho_{A_i}}
\quad\text{for }x\in X_{A_i}\,.
\end{equation}
\end{cor}

\begin{cor}\label{posneg}
If $\A=(A_0,A_1)\in \mm$
then for all $x\in X$,
\begin{equation}\label{posnegsums}
\sum_{n=1}^\infty (f_\A\circ T_{A_0}^n)'(x) >0
\quad\text{and}\quad
\sum_{n=1}^\infty (f_\A\circ T_{A_1}^n)'(x) <0\,,
\end{equation}
and 
\begin{equation}\label{posnegxrhos}
x+\varrho_{A_0}>0
\quad\text{and}\quad
x+\varrho_{A_1}<0\,.
\end{equation}
\end{cor}
\begin{proof}
The inequalities in (\ref{posnegsums}) follow from Lemma \ref{posnegf} (iii),
while (\ref{posnegxrhos}) is an immediate consequence of
(\ref{sumformula}) and
(\ref{posnegsums}).
\end{proof}

\begin{remark}
The inequality $x+\varrho_{A_0}>0$ in (\ref{posnegxrhos}) can also be deduced from the fact that 
$\rho_{A_0}>0$ (by Corollary \ref{rhoposlessminusone}) and
$x\ge0$.
\end{remark}

\section{Sturmian  measures associated to a concave-convex matrix pair}\label{sturmasection}

For $\A\in\mm$, the induced space $X_\A$ becomes an ordered set when equipped
with the usual order on $X=[0,1]$. In particular,
by a \emph{sub-interval} of $X_\A$ we mean any subset of $X_\A$ of the form $I\cap X_\A$
where $I$ is some sub-interval of $X$. Note that a sub-interval of $X_\A$ is
 a sub-interval of $X$ if it is contained in either $X_{A_0}$ or $X_{A_1}$; otherwise it is
a union of two disjoint intervals in $X$.

\begin{defn}
Given a matrix pair $\A \in \mm$,
a closed interval $\Gamma\subset X_\A$ is called 
\emph{$\A$-Sturmian} (or simply \emph{Sturmian})
if $T_\A(\min\Gamma)=T_\A(\max\Gamma)$,
i.e.~its two endpoints $\min \Gamma$ and $\max \Gamma$ have the same image under
the induced dynamical system
$T_\A$.
\end{defn}

\begin{remark}\label{sturmianasturmian}
\item[\, (a)]
The topological conjugacy $h_\A: \Omega\to Y_\A$ (cf.~Remark \ref{conjugacymeasures}) is 
order preserving, so
if $\Gamma\subset X_\A$ is an $\A$-Sturmian interval,
then $h_\A^{-1}(\Gamma\cap Y_\A)$ is
a Sturmian interval 
as defined in Notation \ref{sturmianfullshift}
(i.e. of the form $[0\omega,1\omega]$ for some $\omega\in\Omega$).
\item[\, (b)]
For all $t>0$, an interval is $\A$-Sturmian if and only if it is $\A(t)$-Sturmian.
\end{remark}

\begin{defn}\label{henceforthca}
Let $\ss_\A$ denote the collection of all $\A$-Sturmian intervals.
Note that $\ss_\A$ is naturally parametrized by $X=[0,1]$:
for each $c\in X$ there is a unique $\Gamma\in\ss_\A$ such that 
$T_\A(\min \Gamma)= T_\A(\max\Gamma) =c$.
Henceforth we shall write $c_\A(\Gamma)$ to denote the common value
$T_\A(\min \Gamma)= T_\A(\max\Gamma)$ for an $\A$-Sturmian interval $\Gamma\in\ss_\A$, noting that
\begin{equation*}
c_\A:\ss_\A\to X
\end{equation*}
is a bijection.
As a subset of $X$, we can express $\Gamma\in\ss_\A$ as
\begin{equation}\label{disjunionsturm}
\Gamma  = [T_{A_0}(c_\A(\Gamma)), T_{A_0}(1)] \cup [T_{A_1}(0),T_{A_1}(c_\A(\Gamma))]\,.
\end{equation}
\end{defn}

\begin{remark}\label{neglectsingleton}
It is apparent from (\ref{disjunionsturm}) that, viewed as a subset of $X=[0,1]$, an $\A$-Sturmian interval  $\Gamma$ is
\emph{always} a disjoint union of two closed intervals. Note, however, that for the two extremal cases where $c_\A(\Gamma)=0$ or $1$, one of the intervals in the disjoint union is a singleton set (and the other interval is, respectively, either $X_{A_0}$ or $X_{A_1}$).
These extremal cases are particularly significant, and in the calculations of \S \ref{extsturmsection}
onwards it is convenient to neglect the singleton set, thereby identifying the extremal $\A$-Sturmian interval
with either $X_{A_0}$ or $X_{A_1}$.
\end{remark}

\begin{defn}\label{sturmianmeasdef}
We say that a $T_\A$-invariant Borel probability measure on $X_\A$ is 
\emph{$\A$-Sturmian}
 if its support is contained in 
some $\A$-Sturmian interval.
Let
$\sm_\A$ denote the collection of $\A$-Sturmian measures.
\end{defn}

\begin{remark}
\item[\, (a)]
In view of Remarks \ref{conjugacymeasures} and \ref{sturmianasturmian},
the class of $\A$-Sturmian measures on $X_\A$ is just the $h_\A^*$-image of the class of Sturmian measures
on the shift space $\Omega$,
 i.e.~$\sm_\A = h_\A^*(\mathcal{S})$.
In particular (cf.~Proposition \ref{sturmianclassical}\,(b)), $\sm_\A$ is also naturally parametrized by $X=[0,1]$:  the map $\mathcal{P}\circ (h_\A^*)^{-1}:\sm_\A\to[0,1]$ is a homeomorphism,
and for $\mu\in\sm_\A$ we refer to $\mathcal{P}\circ (h_\A^*)^{-1}(\mu) = \mu(X_{A_1})$
as its \emph{(Sturmian) parameter}.
\item[\, (b)]
For all $t>0$, a measure is $\A$-Sturmian if and only if it is $\A(t)$-Sturmian.
\end{remark}

In \S \ref{technicalsection} we shall identify cases where $\A$-Sturmian measures arise as unique maximizing measures
for $f_{\A(t)}$, $t>0$.
In particular, for certain $t$ the unique $f_{\A(t)}$-maximizing measure is a Sturmian measure of \emph{irrational} parameter, and 
such $\A(t)$ turn out to be \emph{finiteness counterexamples}:

\begin{prop}\label{irrationalcounterexampleconnection}
If $\A\in\mm$ is such that there is a unique $f_\A$-maximizing measure, and this measure is an $\A$-Sturmian measure with irrational parameter $\mathcal{P}$, then $\A$ is a finiteness counterexample
(i.e.~$r(\A) > r(A_{i_1}\cdots A_{i_n})^{1/n}$
for all $n\in\N$ and all choices $i_1,\ldots, i_n\in\{0,1\}$).
In this case 
\begin{equation}
\label{inthelimit3}
r(\A) = \lim_{n\to\infty} r(A_{\omega_1}\cdots A_{\omega_n})^{1/n} \,,
\end{equation} 
holds for the uncountably many Sturmian sequences $\omega=(\omega_n)_{n=1}^\infty$ of parameter $\mathcal{P}$.
\end{prop}
\begin{proof}
By assumption there is a  unique $f_\A$-maximizing measure $\mu$,
and this measure is an $\A$-Sturmian measure with irrational parameter $\mathcal{P}$, which in particular is not a periodic orbit measure. It follows that no $T_\A$-periodic orbit measure is $f_\A$-maximizing, so Proposition
\ref{counterexampleequivalenceprop} implies that $\A$ is a finiteness counterexample,
and that there exist uncountably many sequences $\omega\in\Omega$ such that (\ref{inthelimit3}) holds.
In fact the support of any $\A$-Sturmian measure $\mu$ is uniquely ergodic (see e.g.~\cite[Cor.~1.6]{bouschmairesse}),
so the ergodic theorem holds for the uncountably many points in the support of $\mu$ (see e.g.~\cite[Thm.~6.19]{walters}),
and therefore $\lim_{n\to\infty} \frac{1}{n}\sum_{j=0}^{n-1} f_\A(T_\A^j(p)) =  \int f_\A\, d\mu =  \log r(\A)$
for all $p$ in the support of $\mu$.
Writing $p=h_\A(\omega)$ as in the proof of Proposition \ref{counterexampleequivalenceprop},
the relation (\ref{sumproductexp}) then implies that
(\ref{inthelimit3}) holds for all 
points in the support of the Sturmian measure $m$,
i.e.~for all Sturmian sequences of parameter $\mathcal{P}$.
\end{proof}

\section{The Sturmian transfer function}\label{transfersection}

In order to 
show that the maximizing measure for $f_{\A(t)}$  is supported in some $\A$-Sturmian interval 
$\Gamma\in\ss_\A$,  our strategy will be to add a  coboundary $\phi_\Gamma - \phi_\Gamma\circ T_\A$, 
where the corresponding \emph{Sturmian transfer function} $\phi_\Gamma$ is introduced below,
so that the new function $ f_\A + \phi_\Gamma - \phi_\Gamma\circ T_\A$
takes a constant  value on all of $\Gamma$, and 
 is strictly smaller than this constant value on the complement of $\Gamma$.
 This approach is patterned on ideas of Bousch \cite{bousch} 
  in the setting of the angle-doubling map
 and degree-one trigonometric polynomials.

To proceed, it is convenient to introduce the following:

\begin{defn}
For $\A\in\mm$, to each $\A$-Sturmian interval $\Gamma$ we associate the \emph{hybrid contraction}
 $\tau_\Gamma: X\to X_\A$, defined by
\begin{equation}\label{hybrid}
\tau_\Gamma(x)=
\begin{cases}
T_{A_1}(x)&\text{if }x\in[0,c(\Gamma))\\
T_{A_0}(x)&\text{if }x\in[c(\Gamma),1]\,.
\end{cases}
\end{equation}
\end{defn}

\begin{remark}\label{taugammaremark}
The hybrid contraction $\tau_\Gamma$ satisfies $\tau_\Gamma(X)=\Gamma$,
and is piecewise Lipschitz continuous.
More precisely, its restriction to $[0,c_\A(\Gamma))$ is Lipschitz,
as is its restriction to $[c_\A(\Gamma),1]$.
\end{remark}

\begin{lemma}\label{phiexists}
Given $\A\in \mm$, and
an $\A$-Sturmian interval $\Gamma\in \ss_\A$, 
there exists a unique Lipschitz continuous function $\varphi_{\A,\Gamma}:X\to\mathbb{R}$
which simultaneously satisfies\footnote{The substantial condition is (\ref{defphiformula}),
which determines $\phi_{\A,\Gamma}$ up to an additive constant.
The extra condition
(\ref{zeroconv}) is useful in that it removes any ambiguity when discussing $\phi_{\A,\Gamma}$.}
\begin{equation}\label{defphiformula}
\varphi_{\A, \Gamma}' = \sum_{n=1}^\infty (f_\A\circ \tau_{\Gamma}^n)'
\quad\text{Lebesgue a.e.,}
\end{equation}
and 
\begin{equation}\label{zeroconv}
\varphi_{\A, \Gamma}(0)=0\,.
\end{equation}
\end{lemma}
\begin{proof}
The function $f_\A$ 
is Lipschitz,
and $\tau_\Gamma$ is piecewise Lipschitz (cf.~Remark \ref{taugammaremark}),
so each $\tau_\Gamma^n$ is piecewise Lipschitz, so by Rademacher's Theorem is differentiable Lebesgue
almost everywhere, with $L^\infty$ derivative.
Now
 $\|(\tau_\Gamma^n)'\|_\infty = O(\theta^n)$ as $n\to\infty$ for some $\theta \in(0,1)$,
so the sum
$$\sum_{n=1}^\infty (f_\A\circ \tau_\Gamma^n)'
= \sum_{n=1}^\infty f_\A'\circ \tau_\Gamma^n .( \tau_\Gamma^n)' 
$$
 is Lebesgue almost everywhere convergent (as its $n$th term is
 $O(\theta^n)$),
and  defines an $L^\infty$ 
function with respect to Lebesgue measure on $X$.
In particular, it has a Lipschitz antiderivative $\varphi_\Gamma$,
which is the unique Lipschitz antiderivative up to an additive constant,
hence uniquely defined if it satisfies the additional condition
$\varphi_{\A, \Gamma}(0)=0$.
\end{proof}

\begin{notation}
For $\A\in\mm$, $\Gamma\in\ss_\A$,
the function $\phi_\Gamma=\phi_{\A,\Gamma}$ whose existence and uniqueness is guaranteed
by Lemma \ref{phiexists}
 will be referred to as the corresponding \emph{Sturmian transfer function}.
\end{notation}

\begin{remark}
Note that although the induced function $f_\A$ is only defined on $X_\A$, the Sturmian transfer function
$\phi_\Gamma$ is actually defined on all of $X=[0,1]$. For the most part, however, we shall only
be interested in the restriction of $\phi_\Gamma$ to $X_\A$.
More precisely, we shall be interested in certain properties of $f_\A+\phi_\Gamma$, 
or of $f_\A+ \phi_\Gamma - \phi_\Gamma\circ T_\A$,
considered as functions defined on $X_\A$, beginning with the following Corollary \ref{lipcont}.
\end{remark}

\begin{cor}\label{lipcont}
If $\A\in\mm$, and $\Gamma$ is any $\A$-Sturmian interval,
then both $f_\A+\phi_\Gamma$
and $f_\A+\phi_\Gamma - \phi_\Gamma\circ T_\A$ are Lipschitz continuous functions on $X_\A$.
\end{cor}
\begin{proof}
Both $T_\A$ and $f_\A$ are Lipschitz continuous on $X_\A$, as noted in Remarks \ref{taremarks} and \ref{faremark}, and $\phi_\Gamma$ is Lipschitz continuous on $X$ as noted in Lemma \ref{phiexists}, hence Lipschitz continuous
on $X_\A$. It follows that
both $f_\A+\phi_\Gamma$
and $f_\A+\phi_\Gamma - \phi_\Gamma\circ T_\A$ are Lipschitz continuous on $X_\A$.
\end{proof}

\begin{lemma}\label{flatxai}
Suppose $\A\in\mm$, $t>0$, and $\Gamma$ is any $\A$-Sturmian interval.
The Lipschitz continuous function 
$f_{\A(t)}+\phi_\Gamma - \phi_\Gamma\circ T_\A:X_\A\to\mathbb{R}$ has the property that its restriction to
$\Gamma\cap X_{A_0}$ is a constant function, and
its restriction to
$\Gamma\cap X_{A_1}$ is a constant function.
\end{lemma}
\begin{proof}
By Corollary \ref{lipcont}, the function $f_{\A(t)}+\phi_\Gamma - \phi_\Gamma\circ T_\A$ is Lipschitz continuous on $X_\A$, because $\A(t)\in\mm$.
So by the fundamental theorem of calculus for Lipschitz functions (see e.g.~\cite[Thm.~7.1.15]{kk}),
the required result will follow if it can be shown that
\begin{equation}\label{derivativezero2}
(f_{\A(t)}+\varphi_\Gamma -\varphi_\Gamma \circ T_\A)'=0\quad \text{Lebesgue a.e. on $\Gamma$}\,.
\end{equation}
But $f_{\A(t)}'=f_\A'$, so
(\ref{derivativezero2}) is equivalent to proving that
\begin{equation}\label{derivativezero}
(f_\A +\varphi_\Gamma-\varphi_\Gamma\circ T_\A)'=0\quad \text{Lebesgue a.e. on $\Gamma$}\,.
\end{equation}
To establish this almost everywhere equality,
note that
$$
f_\A' + \phi_\Gamma' = 
f_\A' + \sum_{n=1}^\infty (f_\A\circ \tau_\Gamma^n)' 
=
 \sum_{n=0}^\infty (f_\A\circ \tau_\Gamma^n)' 
=
 \sum_{n=0}^\infty f_\A' \circ  \tau_\Gamma^n. (\tau_\Gamma^n)' 
$$
and 
$$
(\phi_\Gamma\circ T_\A)' =
  \sum_{n=1}^\infty f_\A' \circ \tau_\Gamma^n \circ T_\A .  (\tau_\Gamma^n)'\circ T_\A . T_\A' 
  =  \sum_{n=1}^\infty f_\A' \circ \tau_\Gamma^{n-1}.  (\tau_\Gamma^{n-1})',
$$
since $(\tau_\Gamma\circ T_\A)' = \tau_\Gamma^{n-1}$, so
indeed (\ref{derivativezero}) holds.
\end{proof}

\begin{remark}
In the generality of Lemma \ref{flatxai}, the constant values assumed by
$f_{\A(t)}+\phi_\Gamma - \phi_\Gamma\circ T_\A$ on
$\Gamma\cap X_{A_0}$ and
$\Gamma\cap X_{A_1}$ 
do not coincide.
However, we shall shortly give (see Lemma \ref{flattenlemma}) an extra condition
which does ensure
that $f_{\A(t)}+\phi_\Gamma - \phi_\Gamma\circ T_\A$ takes the \emph{same} constant value
on the whole of $\Gamma$. Indeed this possibility is a key tool in our strategy.
\end{remark}

\section{The extremal Sturmian intervals}\label{extsturmsection}

\subsection{Formulae involving extremal intervals}

As noted in Remark \ref{neglectsingleton}, an $\A$-Sturmian interval is the disjoint union of two closed intervals when viewed as a subset of $X=[0,1]$. However, the two \emph{extremal} cases yield a leftmost
$\A$-Sturmian interval equal to $X_{A_0}\cup\{T_{A_1}(0)\}$,
and a rightmost $\A$-Sturmian interval equal to $\{T_{A_0}(1)\}\cup X_{A_1}$.
The presence of singleton sets in these expressions is notationally inconvenient, and unnecessary
for our purposes, so henceforth we neglect them.

More precisely, henceforth
the leftmost $\A$-Sturmian interval is taken to be $X_{A_0}=T_{A_0}(X)$, and denoted
by $\Gamma_0$, so that $\tau_{\Gamma_0}=T_{A_0}$;
the rightmost $\A$-Sturmian interval is taken to be $X_{A_1}=T_{A_1}(X)$, and denoted
by $\Gamma_1$, so that $\tau_{\Gamma_1}=T_{A_1}$.

When the $\A$-Sturmian interval $\Gamma$ is either $\Gamma_0$ or $\Gamma_1$, there is an explicit formula for the Sturmian transfer function $\phi_\Gamma$:

\begin{lemma}\label{easyphiformulae}
Suppose $\A\in\mm$.
For $i\in\{0,1\}$, and all $x\in X$,
\begin{equation}\label{phiiformulae}
\phi_{\Gamma_i}(x) =  \log \left( \frac{x+\rho_{A_i}}{\rho_{A_i}} \right) \,.
\end{equation}
\end{lemma}
\begin{proof}
Now $\tau_{\Gamma_i} = T_{A_i}$, so the defining formula (\ref{defphiformula}) becomes
\begin{equation}\label{defphiiformula}
\phi_{\Gamma_i}'(x) = \sum_{n=1}^\infty (f_\A\circ T_{A_i}^n)'(x)\,,
\end{equation}
and then (\ref{sumformula}) implies that
\begin{equation*}
\phi_{\Gamma_i}'(x) = \frac{1}{x+\varrho_{A_i}}\,.
\end{equation*}
Noting that the sign of $x+\varrho_{A_i}$ is positive when $i=0$ and negative when $i=1$
(see Corollary \ref{posneg}), as well as the convention that $\phi_{\Gamma_i}(0)=0$ (see Definition \ref{phiexists}),
we deduce the required expression (\ref{phiiformulae}).
\end{proof}

\begin{defn}\label{Deltadefndefn}
Given $\A=(A_0,A_1)\in \mm$ and
$\Gamma\in\ss_\A$, 
define $\Delta_\A(\Gamma)  \in \mathbb{R}$ by
\begin{equation}\label{Deltadefn}
\Delta_\A(\Gamma)
=
\left(\phi_\Gamma(1)-\phi_\Gamma(0) \right)
-
\left( \phi_\Gamma(T_{A_0}(1))-\phi_\Gamma(T_{A_1}(0)) \right) \,,
\end{equation}
noting the equivalent expression
\begin{equation}\label{Deltadefnalt}
\Delta_\A(\Gamma)
=
\phi_\Gamma(1)
-
\left( \phi_\Gamma(T_{A_0}(1))-\phi_\Gamma(T_{A_1}(0)) \right) 
\end{equation}
as a consequence of the convention that $\phi_\Gamma(0)=0$
(see Lemma \ref{phiexists}).
\end{defn}

The values $\Delta_\A(\Gamma_i)$ play an important role, so it will be useful to record 
the following explicit formulae:

\begin{lemma}\label{deltagammailemma}
Suppose $\A\in\mm$.
For $i\in\{0,1\}$,
\begin{equation}\label{deltagammai}
\Delta_\A(\Gamma_i) = 
\log \left( 
\frac{(1+\rho_{A_i}) \left(\frac{b_1}{b_1+d_1}+\rho_{A_i}\right)}{ \rho_{A_i} \left( \frac{a_0}{a_0+c_0}+\rho_{A_i}\right)}
\right) \,.
\end{equation}
\end{lemma}
\begin{proof}
This is immediate from the defining formula (\ref{Deltadefn})
(or (\ref{Deltadefnalt})) for $\Delta_\A(\Gamma_i)$, together with formula
(\ref{phiiformulae}) for $\phi_{\Gamma_i}$, and the fact that $T_{A_0}(1)=\frac{a_0}{a_0+c_0}$
and $T_{A_1}(0) = \frac{b_1}{b_1+d_1}$.
\end{proof}

\begin{cor}\label{deltasigns}
If $\A\in\mm$ then
\begin{equation*}
\label{starrr}
\Delta_\A(\Gamma_1) < 0 < \Delta_\A(\Gamma_0)\,.
\end{equation*}
\end{cor}
\begin{proof}
The four terms
$
\rho_{A_i}
$,
$
1+\rho_{A_i}
$,
$
\frac{a_0}{a_0+c_0}+\rho_{A_i}
$,
$
\frac{b_1}{b_1+d_1}+\rho_{A_i}
$
in (\ref{deltagammai}) are all positive if $i=0$, and all negative if $i=1$, by 
the inequalities (\ref{posnegxrhos})
in Corollary \ref{posneg}.
Now $\A\in\mm$ implies that (\ref{definequal1}) holds, so
$
\frac{b_1}{b_1+d_1}+\rho_{A_i}
>
\frac{a_0}{a_0+c_0}+\rho_{A_i}
$,
and clearly $1+\rho_{A_i}>\rho_{A_i}$
Consequently
$$
\frac{(1+\rho_{A_i}) \left(\frac{b_1}{b_1+d_1}+\rho_{A_i}\right)}{ \rho_{A_i} \left( \frac{a_0}{a_0+c_0}+\rho_{A_i}\right)}
$$
is strictly greater than $1$ if $i=0$, and strictly smaller than $1$ if $i=1$.
The result then follows from Lemma \ref{deltagammailemma}.
\end{proof}

\subsection{Adaptations for other matrix pairs}\label{adaptations}
As mentioned in \S \ref{relatsubsection}, the methods of this paper can be adapted 
so as to give alternative proofs 
of certain results
(analogues of Theorem \ref{maintheorem})
 mentioned in \S \ref{generalsection},
namely establishing that a full Sturmian family is generated
by the matrix pair (\ref{standardpair}), and for matrix pairs corresponding to sub-cases of (\ref{bmfamily}) and (\ref{kozyakinpairs})
which lie in the boundary of $\nn$.\footnote{Note that all
of the matrix pairs in (\ref{bmfamily}),   (\ref{kozyakinpairs}), (\ref{standardpair}) have the property that
$A_0$ is projectively concave and $A_1$ is projectively convex.}
In this subsection we indicate the modifications necessary to handle these cases.

Firstly, the induced space $X_\A$ may be the whole of $X=[0,1]$ rather than a disjoint union of two closed intervals:
this occurs if $a_0/c_0=b_1/d_1$ (i.e.~when (\ref{definequal1}) becomes an equality), which is the case for the pair
(\ref{standardpair}), and for (\ref{kozyakinpairs}) if $bc=1$.

Secondly, in each of the cases (\ref{bmfamily}), (\ref{kozyakinpairs}) and (\ref{standardpair}), the 
induced maps $T_{A_0}$ and $T_{A_1}$ have fixed points
at
0 and 1 respectively, so that the dynamical system $T_\A$ also fixes these points.
For (\ref{standardpair}), both 0 and 1 are indifferent fixed points, i.e.~$T_{A_0}'(0)=1=T_{A_1}'(1)$.
For   (\ref{bmfamily}) and (\ref{kozyakinpairs}) these fixed points are unstable for the induced maps $T_{A_0}$ and $T_{A_1}$, i.e.~$T_{A_0}'(0)>1$ and $T_{A_1}'(1)>1$, but both of these maps also have stable fixed points in the interior of $X=[0,1]$.
Consequently for (\ref{standardpair}) the dynamical system $T_\A:X\to X$ has indifferent fixed points at 0 and 1,
and no other fixed points, while for (\ref{bmfamily}) and (\ref{kozyakinpairs}) the dynamical system $T_\A$ has stable fixed points at 0 and 1, and two further unstable fixed points in the interior of $X$.

The potentially problematic stable fixed points for $T_\A$ can in fact be avoided by omitting to consider 
the two extremal $\A$-Sturmian intervals: this ensures the
asymptotic $\|(\tau_\Gamma^n)'\|_\infty = O(\theta^n)$ as $n\to\infty$,
 $\theta \in(0,1)$,  and the existence of Sturmian transfer functions is proved as in Lemma \ref{phiexists}.
 In the case where $T_\A$ has indifferent fixed points, it is even possible to consider 
 extremal $\A$-Sturmian intervals, as the series defining the Sturmian transfer function is nonetheless convergent.
 The existence of Sturmian transfer functions then allows the
 remainder of the method of proof to proceed essentially as for matrix pairs in $\nn$,
 ultimately establishing analogues of the main result Theorem \ref{maintheorem}.

\section{Associating $\A$-Sturmian intervals to parameter values}\label{particularsection}

\begin{notation}
For a Sturmian interval $\Gamma\in\ss_\A$,
let $s_\Gamma\in\sm_\A$
denote the $\A$-Sturmian measure supported by $\Gamma$, i.e.~$s_\Gamma$ is
the unique $T_\A$-invariant probability measure whose support is contained in $\Gamma$.
\end{notation}

\begin{lemma}\label{flattenlemma}
Suppose $\A\in \mm$.
If $t\in\mathbb{R}^+$
and
$\Gamma\in \ss_\A$  are such that
\begin{equation}\label{keyflatteningequation}
f_{\A(t)}(T_{A_0}(1)) - f_{\A(t)}(T_{A_1}(0))
= \Delta_\A( \Gamma)\,,
\end{equation}
then the Lipschitz continuous function
$f_{\A(t)}+\phi_{\Gamma} - \phi_{\Gamma}\circ T_\A$ is
  equal to the constant value
$\int f_{\A(t)}\, ds_\Gamma$ when restricted to $\Gamma$.
\end{lemma}
\begin{proof}
By Lemma \ref{flatxai} we know that
$f_{\A(t)}+\varphi_\Gamma-\varphi_\Gamma\circ T_\A$ 
is constant when restricted to
$\Gamma\cap X_{A_0}$, and also constant
when restricted to $\Gamma\cap X_{A_1}$. 
To prove that these constant values are the \emph{same},
it suffices to show that $f_{\A(t)}+\phi_\Gamma - \phi_\Gamma\circ T_\A$ 
takes the same value at the point $T_{A_0}(1)\in X_{A_0}$ as
it does at the point $T_{A_1}(0)\in X_{A_1}$.
But the equality
$$
\left(f_{\A(t)}+\phi_\Gamma - \phi_\Gamma\circ T_\A\right)(T_{A_0}(1))
=
\left(f_{\A(t)}+\phi_\Gamma - \phi_\Gamma\circ T_\A\right)(T_{A_1}(0))
$$
holds if and only if
$$
f_{\A(t)}(T_{A_0}(1)) - f_{\A(t)}(T_{A_1}(0)) = \phi_\Gamma(1)-\phi_\Gamma(0) - \left(\phi_\Gamma(T_{A_0}(1)) - \phi_\Gamma(T_{A_1}(0))\right)\,,
$$
in other words
$f_{\A(t)}(T_{A_0}(1)) - f_{\A(t)}(T_{A_1}(0)) = \Delta_\A( \Gamma)$, which is precisely
the  hypothesis
(\ref{keyflatteningequation}).
\end{proof}

\begin{cor}\label{flattencor}
Given $\A\in \mm$, if $t\in\mathbb{R}^+$
and
$\Gamma\in \ss_\A$  are such that
\begin{equation}\label{keyflatteningequation1}
\log \left( \left(\frac{a_0+c_0}{b_1+d_1}\right) t^{-1} \right)
= \Delta_\A(\Gamma) \,,
\end{equation}
then the Lipschitz continuous function
$f_{\A(t)}+\phi_{\Gamma} - \phi_{\Gamma}\circ T_\A$ is
  equal to the constant value
$\int f_{\A(t)}\, ds_\Gamma$ on $\Gamma$.
\end{cor}
\begin{proof}
By Lemma \ref{flattenlemma} it suffices to show that
$$f_{\A(t)}(T_{A_0}(1)) - f_{\A(t)}(T_{A_1}(0))
=
\log \left( \left(\frac{a_0+c_0}{b_1+d_1}\right) t^{-1} \right)\,,$$
and by
Lemma \ref{simple}(iii) this is equivalent to showing that
$$f_{\A}(T_{A_0}(1)) - f_{\A}(T_{A_1}(0))
=
\log \left( \frac{a_0+c_0}{b_1+d_1} \right)\,.$$
Substituting $T_{A_0}(1)=\frac{a_0}{a_0+c_0}$ and $T_{A_1}(0)=\frac{b_1}{b_1+d_1}$
into, respectively, the formulae (\ref{explicitfalt}) for $f_\A$ on $X_{A_0}$ and $X_{A_1}$ yields
\begin{equation}\label{fata01}
f_\A(T_{A_0}(1)) = \log(a_0+c_0)
\end{equation}
and
\begin{equation}\label{fata10}
 f_\A(T_{A_1}(0)) = \log(b_1+d_1)\,,
 \end{equation}
 so the result follows.
\end{proof}

In view of equation
(\ref{keyflatteningequation1})
we make the following definition:

\begin{defn}
For $\A\in\mm$ and $i\in\{0,1\}$, define $t_i=t_i(\A)$ by
\begin{equation}\label{tiadef}
t_i = t_i(\A) =  \left(\frac{a_0+c_0}{b_1+d_1}\right) e^{-\Delta_\A(\Gamma_i)}\,,
\end{equation}
so that
\begin{equation*}
\log \left( \left(\frac{a_0+c_0}{b_1+d_1}\right) t_i^{-1} \right)
= \Delta_\A(\Gamma_i)\,.
\end{equation*}
\end{defn}

\begin{remark}
Since $e^{-\Delta_\A(\Gamma_0)} < 1 < e^{-\Delta_\A(\Gamma_1)}$
by (\ref{starrr}), it follows that 
\begin{equation*}
t_0(\A) <t_1(\A)\,.
\end{equation*}
\end{remark}

\begin{lemma}
For $\A\in\mm$ and 
$i\in\{0,1\}$,
\begin{equation}\label{initialtiaexpressions}
t_i(\A)
=
\frac{\rho_{A_i} \left( a_0 + \rho_{A_i}(a_0+c_0) \right)}{(1+\rho_{A_i})\left(b_1+\rho_{A_i}(b_1+d_1)\right)} \,.
\end{equation}
\end{lemma}
\begin{proof}
From (\ref{deltagammai}) we see that for $i\in\{0,1\}$,
\begin{equation*}
e^{-\Delta_\A(\Gamma_i)} = 
\frac{ \rho_{A_i} \left( \frac{a_0}{a_0+c_0}+\rho_{A_i}\right)}{(1+\rho_{A_i}) \left(\frac{b_1}{b_1+d_1}+\rho_{A_i}\right)}\,,
\end{equation*}
so that (\ref{tiadef}) gives
\begin{equation}\label{firsttexpression}
t_i(\A) = 
\left(\frac{a_0+c_0}{b_1+d_1}\right) e^{-\Delta_\A(\Gamma_i)}
=
\frac{\rho_{A_i} \left( a_0 + \rho_{A_i}(a_0+c_0) \right)}{(1+\rho_{A_i})\left(b_1+\rho_{A_i}(b_1+d_1)\right)} \,,
\end{equation}
which is the required expression (\ref{initialtiaexpressions}).
\end{proof}

A consequence is the following property:

\begin{cor}
For $\A\in \mm$,  $t\in\mathbb{R}^+$, and $i\in\{0,1\}$,
\begin{equation}\label{tti}
t_i( \A(t)) = \frac{t_i(\A)}{t}\,.
\end{equation}
\end{cor}
\begin{proof}
This follows easily from
(\ref{initialtiaexpressions}),
and the easily verified fact (used only in the proof of the $i=1$ case) that 
$\rho_{tA_1}=\rho_{A_1}$.
Specifically, for $i\in\{0,1\}$,
\begin{equation*}
t_i(\A(t))
=
\frac{\rho_{A_i} \left( a_0 + \rho_{A_i}(a_0+c_0) \right)}{(1+\rho_{A_i})\left(tb_1+\rho_{A_i}(tb_1+td_1)\right)} 
=\frac{1}{t}
\frac{\rho_{A_i} \left( a_0 + \rho_{A_i}(a_0+c_0) \right)}{(1+\rho_{A_i})\left(b_1+\rho_{A_i}(b_1+d_1)\right)} 
= \frac{t_i(\A)}{t}\,.
\end{equation*}
\end{proof}

\begin{lemma}
For $\A\in\mm$, the quantities $t_0(\A)$ and $t_1(\A)$ admit the following alternative expressions:
\begin{equation}\label{tiaexpressions}
t_0(\A)
=
\frac{\det A_0}{ \left( a_0 - b_0 ( 1+\rho_{A_0}^{-1})\right) \left( d_1 +b_1(1+\rho_{A_0}^{-1})\right)}
\end{equation}
and
\begin{equation}\label{tiaexpressionsi1}
t_1(\A)
=
\frac{ \left(a_0+c_0(1+\rho_{A_1}^{-1})^{-1}\right) \left(a_1-b_1(1+\rho_{A_1}^{-1})\right)}{\det A_1}\,.
\end{equation}
\end{lemma}
\begin{proof}
Since (\ref{firsttexpression}) implies
\begin{equation}\label{intermediatetexpression}
t_0(\A) = 
\frac{a_0 + \rho_{A_0}(a_0+c_0) }{(1+\rho_{A_0})\left(d_1+b_1(1+\rho_{A_0}^{-1})\right)} \,,
\end{equation}
we see that $t_0(\A)$ is equal to (\ref{tiaexpressions}) if and only if
\begin{equation}\label{ifftexpression}
\frac{a_0 + \rho_{A_0}(a_0+c_0) }{1+\rho_{A_0}}
=
\frac{a_0d_0-b_0c_0}{a_0-b_0(1+\rho_{A_0}^{-1})}\,.
\end{equation}
Clearing fractions in (\ref{ifftexpression}) reveals it to be equivalent to the equation
$$
q_{A_0}(\rho_{A_0}) = \alpha_{A_0} \rho_{A_0}^2+\beta_{A_0}\rho_{A_0}-b_0 =0\,,
$$
which is true by Lemma \ref{rhoquadratic}.

Since (\ref{firsttexpression}) implies
\begin{equation}\label{intermediatetexpression1}
t_1(\A) = 
\frac{\rho_{A_1}\left( a_0+c_0(1+\rho_{A_1}^{-1})^{-1}\right)}{b_1+\rho_{A_1}(b_1+d_1)} \,,
\end{equation}
we see that $t_1(\A)$ is equal to (\ref{tiaexpressionsi1}) if and only if
\begin{equation}\label{ifftexpression1}
\frac{\rho_{A_1}}{b_1+\rho_{A_1}(b_1+d_1)}
=
\frac{a_1-b_1(1+\rho_{A_1}^{-1})}{\det A_1}\,.
\end{equation}
Clearing fractions in (\ref{ifftexpression1}) reveals it to be equivalent to the equation
$$
q_{A_1}(\rho_{A_1}) = \alpha_{A_1} \rho_{A_1}^2+\beta_{A_1}\rho_{A_1}-b_1 =0\,,
$$
which is true by Lemma \ref{rhoquadratic}.
\end{proof}

\begin{notation}
For $\A\in\mm$, let $\t_\A$ denote the open interval $\left(t_0(\A), t_1(\A)\right)$.
\end{notation}

\begin{prop}\label{analogue} 
Let $\A\in\mm$.
For each $t \in \t_\A$ there exists 
an $\A$-Sturmian interval $\Gamma_\A(t)\in \ss_\A$ such that
$f_{\A(t)}+\phi_{\Gamma_t} - \phi_{\Gamma_t}\circ T_\A$ is 
  equal to the constant value
$\int f_{\A(t)}\, ds_{\Gamma_\A(t)}$ on $\Gamma_\A(t)$.
\end{prop}
\begin{proof} 
First we show that $\Delta_\A: \Gamma \mapsto \Delta_\A(\Gamma)$ is continuous.
The formula (\ref{Deltadefnalt}) defines
\begin{align*}
\Delta_\A(\Gamma)
& =
\phi_\Gamma(1)
-
\phi_\Gamma(T_{A_0}(1)) + \phi_\Gamma(T_{A_1}(0)) \cr
& =
\phi_\Gamma(1)
-
\phi_\Gamma\left(\frac{a_0}{a_0+c_0}\right) + \phi_\Gamma\left(\frac{b_1}{b_1+d_1}\right) \,,
\end{align*}
so the continuity of $ \Delta_\A$ will follow from the fact
that $\Gamma\mapsto \phi_\Gamma(z)$ is continuous for each $z\in X$.
To see this, first note that Definition \ref{phiexists} gives
\begin{equation*}
\phi_\Gamma(z)
= \phi_\Gamma(z)-\phi_\Gamma(0) 
 =\int_0^z \phi_\Gamma' 
 = \int_0^z \sum_{n=1}^\infty (f_\A \circ \tau_\Gamma^n)' \,,
\end{equation*}
and re-writing this integral as
\begin{equation*}
  \sum_{n=1}^\infty \int_0^z (f_\A \circ \tau_\Gamma^n)' 
 =  \sum_{n=1}^\infty \int_{\tau^n_\Gamma[0,z]} f_\A' 
 = \sum_{n=1}^\infty \int  \mathbbm{1}_{\tau^n_\Gamma[0,z]}  f_\A'
 =\int f_\A' \sum_{n=1}^\infty \mathbbm{1}_{\tau^n_\Gamma[0,z]}
\end{equation*}
gives
\begin{equation}\label{phigammaexpression}
\phi_\Gamma(z) = \int f_\A' H_z(\Gamma) \,,
\end{equation}
where 
\begin{equation*}
H_z(\Gamma) = \sum_{n=1}^\infty \mathbbm{1}_{\tau^n_\Gamma[0,z]} \,.
\end{equation*}
Now each map $H_{z,n}:\Gamma\mapsto \mathbbm{1}_{\tau^n_\Gamma[0,z]}$ 
clearly belongs to $C([\Gamma_0,\Gamma_1],L^1)$, the space of
continuous functions from $[\Gamma_0,\Gamma_1]$ to $L^1=L^1(dx)$,
and $\sum_{n=1}^\infty H_{z,n}$
is convergent in
$C([\Gamma_0,\Gamma_1],L^1)$, so $H_z(\cdot) \in C([\Gamma_0,\Gamma_1],L^1)$.
It then follows from (\ref{phigammaexpression}) that 
$\Gamma\mapsto \phi_\Gamma(z)$ is continuous, as required.

Now 
note that
the function $G_\A$ defined by
\begin{equation}\label{gadef}
G_\A(t) = \log \left( \left(\frac{a_0+c_0}{b_1+d_1}\right) t^{-1} \right)
\end{equation}
is strictly decreasing, since $a_0,c_0,b_1,d_1>0$,
so if $t \in \t_\A = (t_0(\A),t_1(\A))$ then 
\begin{equation}\label{ginside}
G_\A(t)\in \left(G_\A(t_1(\A)),G_\A(t_0(\A))\right) = (\Delta_\A(\Gamma_1),\Delta_\A(\Gamma_0))\,.
\end{equation}
Now $\Delta_\A$ is continuous, so applying
the intermediate value theorem to this function (defined on the interval $[\Gamma_0,\Gamma_1]$) we see that in view of (\ref{ginside}), there exists an $\A$-Sturmian interval, which we denote by $\Gamma_\A(t)$, such that 
$\Gamma_\A(t)\in (\Gamma_0,\Gamma_1)$
and 
\begin{equation}\label{deltatdefeq}
\Delta_\A(\Gamma_\A(t))=G_\A(t)\,.
\end{equation}
In other words,
\begin{equation*}
\log \left( \left(\frac{a_0+c_0}{b_1+d_1}\right) t^{-1} \right)
= \Delta_\A(\Gamma_\A(t)) \,,
\end{equation*}
so that Corollary \ref{flattencor} implies that 
$f_{\A(t)}+\phi_{\Gamma_t} - \phi_{\Gamma_t}\circ T_\A
= \int f_{\A(t)}\, ds_{\Gamma_\A(t)}$ on $\Gamma_\A(t)$, as required.
\end{proof}

\section{The case when one matrix dominates}\label{dominatessec}

It will be useful to record the value of the induced function
$f_\A$ at the two fixed points of $T_\A$:

\begin{lemma}\label{fapai}
For $\A\in\mm$ and $i\in\{0,1\}$,
\begin{equation*}
f_\A(p_{A_i})
=
\log\left( \frac{\det A_i}{a_i - b_i ( 1+\rho_{A_i}^{-1})}
\right)
=
\log\left( \frac{\det A_i}{a_i-b_i -\frac{1}{2}(\beta_{A_i}+\gamma_{A_i})}\right)
\,.
\end{equation*}
\end{lemma}
\begin{proof}
Straightforward computation using (\ref{rhoaredone}),
(\ref{paformularedone}), 
and (\ref{explicitfalt}).
\end{proof}

We first consider a sufficient condition for 
the projectively concave matrix
$A_0$ to be the dominant matrix of the pair $\A=(A_0,A_1)$:

\begin{theorem}\label{t0larger1theorem}
If $\A\in\mm$ is such that
\begin{equation}\label{t0larger1}
t_0(\A) \ge 1\,,
\end{equation}
then the Dirac measure at the fixed point $p_{A_0}$ is the unique $f_\A$-maximizing measure;
in particular, the joint spectral radius of $\A$ is equal to the spectral radius of $A_0$.
\end{theorem}
\begin{proof}
Choosing 
$\phi(x) = \phi_{\Gamma_0}(x) = \log \left(\frac{x+\rho_{A_0}}{\rho_{A_0}}\right)$
ensures,
by Lemma \ref{flatxai}, 
 that
$f_\A+\phi-\phi\circ T_\A$ is constant when restricted to $X_{A_0} = \Gamma_0$,
and the constant value assumed by this function is clearly $f_\A(p_{A_0})$.
The result will follow if we can show that $f_\A+\phi-\phi\circ T_\A$ is strictly decreasing
on $X_{A_1}$,
and that the value $(f_\A+\phi-\phi\circ T_\A)(\frac{b_1}{b_1+d_1})$
at the left endpoint of $X_{A_1}$ is no greater than the constant value $f_\A(p_{A_0})$.
This is because
the Dirac measure $\delta_{p_{A_0}}$ will then clearly be the unique maximizing measure for
$f_\A+\phi-\phi\circ T_\A$, and hence the unique maximizing measure for $f_\A$.

To compute the value $(f_\A+\phi-\phi\circ T_\A)(\frac{b_1}{b_1+d_1})$ we recall
from (\ref{fata10}) that
$$
f_\A\left( \frac{b_1}{b_1+d_1}\right)
=  f_\A(T_{A_1}(0)) = \log(b_1+d_1)\,,
$$
and note that
$$
\phi\left(T_\A\left(\frac{b_1}{b_1+d_1}\right)\right)=\phi(0) = 0
\,,
$$
and 
$$
\phi\left( \frac{b_1}{b_1+d_1}\right)=\log\left( \frac{ \frac{b_1}{b_1+ d_1} +\rho_{A_0}}{\rho_{A_0}}\right)
=\log\left(\frac{\left( d_1 +b_1(1+\rho_{A_0}^{-1})\right)}{b_1+d_1}\right)
\,.
$$
Therefore
\begin{equation}\label{endpointformula}
(f_\A+\phi-\phi\circ T_\A)\left(\frac{b_1}{b_1+d_1}\right)
=
\log \left( d_1 +b_1(1+\rho_{A_0}^{-1})\right)
 \,.
\end{equation}

By Lemma \ref{fapai},
\begin{equation}\label{fpt0}
f_\A(p_{A_0})
=
\log\left( \frac{\det A_0}{a_0 - b_0 ( 1+\rho_{A_0}^{-1})}
\right)\,,
\end{equation}
so (\ref{endpointformula}) and (\ref{fpt0}) imply that the desired inequality
$$(f_\A+\phi-\phi\circ T_\A)\left(\frac{b_1}{b_1+d_1}\right)\le f_\A(p_{A_0})$$
is precisely the hypothesis (\ref{t0larger1}),
since
\begin{equation*}
t_0(\A)
=
\frac{\det A_0}{ \left( a_0 - b_0 ( 1+\rho_{A_0}^{-1})\right) \left( d_1 +b_1(1+\rho_{A_0}^{-1})\right)}
\end{equation*}
by (\ref{tiaexpressions}).

It remains to show that $f_\A+\phi-\phi\circ T_\A$ is strictly decreasing on $X_{A_1}$.
Suppose $x\in X_{A_1}$. We know by (\ref{explicitfalt}) that
$$
f_\A(x) =  \log  \left( \frac{\det A_1}{-\alpha_{A_1}(x+ \sigma_{A_1})} \right)\,.
$$
Now 
$$
\phi(x) = \log \left( \frac{x+\rho_{A_0}}{\rho_{A_0}}\right)\,,
$$
so 
$$
\phi(T_\A(x)) 
=
\log \left( \frac{S_{A_1}(x)+\rho_{A_0}}{\rho_{A_0}}\right)\,,
$$
and therefore
$$
(f_\A+\phi-\phi\circ T_\A)(x) 
=
\log\left( \frac{\det A_1 (x+\rho_{A_0})}{-\alpha_{A_1}(x+\sigma_{A_1})(S_{A_1}(x)+\rho_{A_0})}\right)\,.
$$
It therefore suffices to show that
\begin{equation}\label{suffdec}
x\mapsto \frac{x+\rho_{A_0}}{-\alpha_{A_1}(x+\sigma_{A_1})(S_{A_1}(x)+\rho_{A_0})}
\end{equation}
is strictly decreasing.
For this note that
$$
S_{A_1}(x)+\rho_{A_0}
= \frac{(b_1+d_1)x-b_1}{-\alpha_{A_1}(x+\sigma_{A_1})}+\rho_{A_0}
=
\frac{ (b_1+d_1-\alpha_{A_1}\rho_{A_0})x +(a_1-b_1)\rho_{A_0}-b_1}{-\alpha_{A_1}(x+\sigma_{A_1})}
$$
so (\ref{suffdec}) is seen to be the M\"obius function
\begin{equation*}
x\mapsto \frac{x+\rho_{A_0}}{ (b_1+d_1-\alpha_{A_1}\rho_{A_0})x +(a_1-b_1)\rho_{A_0}-b_1} \,,
\end{equation*}
which is known to be strictly decreasing by Lemma \ref{techderivlemma}.
\end{proof}

As a consequence of Theorem \ref{t0larger1theorem} we obtain:

\begin{cor}\label{t0largerttheorem}
If $\A\in\mm$ and $t\in \mathbb{R}^+$ are such that
\begin{equation}\label{t0largert}
t \le t_0(\A) \,,
\end{equation}
then the Dirac measure at the fixed point $p_{A_0}$ is the unique $f_{\A(t)}$-maximizing measure;
in particular, the joint spectral radius of $\A(t)$ is equal to the spectral radius of $A_0$.
\end{cor}
\begin{proof}
The assumption (\ref{t0largert}) means, using (\ref{tti}),
that $t_0(\A(t))\ge 1$, so the result follows by applying
Theorem \ref{t0larger1theorem} 
with $\A$ replaced by $\A(t)$.
\end{proof}

We now turn to an analogous sufficient condition for the projectively convex matrix $A_1$
to be dominant:

\begin{theorem}\label{t1smaller1theorem}
If $\A\in\mm$ is such that
\begin{equation}\label{t1smaller1}
t_1(\A) \le 1\,,
\end{equation}
then the Dirac measure at the fixed point $p_{A_1}$ is the unique $f_\A$-maximizing measure;
in particular, the joint spectral radius of $\A$ is equal to the spectral radius of $A_1$.
\end{theorem}
\begin{proof}
Choosing 
$\phi(x) = \phi_{\Gamma_1}(x) = \log \left(\frac{x+\rho_{A_1}}{\rho_{A_1}}\right)$
ensures,
by Lemma \ref{flatxai}, 
 that
$f_\A+\phi-\phi\circ T_\A$ is constant when restricted to $X_{A_1} = \Gamma_1$,
and the constant value assumed by this function is clearly $f_\A(p_{A_1})$.
The result will follow if we can show that $f_\A+\phi-\phi\circ T_\A$ is strictly increasing
on $X_{A_0}$,
and that the value $(f_\A+\phi-\phi\circ T_\A)(\frac{a_0}{a_0+c_0})$
at the right endpoint of $X_{A_0}$ is no greater than the constant value $f_\A(p_{A_1})$.
This is because
the Dirac measure $\delta_{p_{A_1}}$ will then clearly be the unique maximizing measure for
$f_\A+\phi-\phi\circ T_\A$, and hence the unique maximizing measure for $f_\A$.

To compute the value $(f_\A+\phi-\phi\circ T_\A)(\frac{a_0}{a_0+c_0})$ we recall
from (\ref{fata01}) that
$$
f_\A\left( \frac{a_0}{a_0+c_0}\right)
=  f_\A(T_{A_0}(1)) = \log(a_0+c_0)\,,
$$
and note that
$$
\phi\left(T\left(\frac{a_0}{a_0+c_0}\right)\right)=\phi(1) = \log\left( \frac{1+\rho_{A_1}}{\rho_{A_1}}\right)
= \log(1+\rho_{A_1}^{-1})
\,,
$$
and 
$$
\phi\left( \frac{a_0}{a_0+c_0}\right)=\log\left( \frac{ \frac{a_0}{a_0+c_0} +\rho_{A_1}}{\rho_{A_1}}\right)
=\log\left(\frac{\left( c_0 +a_0(1+\rho_{A_1}^{-1})\right)}{a_0+c_0}\right)
\,.
$$
Therefore
\begin{equation}\label{endpointformula2}
(f_\A+\phi-\phi\circ T_\A)\left(\frac{a_0}{a_0+c_0}\right)
=
\log \left( a_0 +c_0(1+\rho_{A_1}^{-1})^{-1} \right)
 \,.
\end{equation}

By Lemma \ref{fapai},
\begin{equation}\label{fpt02}
f_\A(p_{A_1})
=
\log\left( \frac{\det A_1}{a_1 - b_1 ( 1+\rho_{A_1}^{-1})}
\right)\,,
\end{equation}
so (\ref{endpointformula2}) and (\ref{fpt02}) imply that the desired inequality
$$(f_\A+\phi-\phi\circ T_\A)\left(\frac{a_0}{a_0+c_0}\right)\le f_\A(p_{A_1})$$
is precisely the hypothesis (\ref{t1smaller1}),
since
\begin{equation*}
t_1(\A)
=
\frac{ \left(a_0+c_0(1+\rho_{A_1}^{-1})^{-1}\right) \left(a_1-b_1(1+\rho_{A_1}^{-1})\right)}{\det A_1}
\end{equation*}
by (\ref{tiaexpressions}).

It remains to show that $f_\A+\phi-\phi\circ T_\A$ is strictly increasing on $X_{A_0}$.
Suppose $x\in X_{A_0}$. We know by (\ref{explicitfalt}) that
$$
f_\A(x) =  \log  \left( \frac{\det A_0}{-\alpha_{A_0}(x+ \sigma_{A_0})} \right)\,.
$$
Now 
$$
\phi(x) = \log \left( \frac{x+\rho_{A_1}}{\rho_{A_1}}\right)\,,
$$
so 
$$
\phi(T_\A(x)) 
=
\log \left( \frac{S_{A_0}(x)+\rho_{A_1}}{\rho_{A_1}}\right)\,,
$$
and therefore
$$
(f_\A+\phi-\phi\circ T_\A)(x) 
=
\log\left( \frac{\det A_0 (x+\rho_{A_1})}{-\alpha_{A_0}(x+\sigma_{A_0})(S_{A_0}(x)+\rho_{A_1})}\right)\,.
$$
It therefore suffices to show that
\begin{equation}\label{suffinc}
x\mapsto \frac{x+\rho_{A_1}}{-\alpha_{A_0}(x+\sigma_{A_0})(S_{A_0}(x)+\rho_{A_1})}
\end{equation}
is strictly increasing.
For this note that
$$
S_{A_0}(x)+\rho_{A_1}
= \frac{(b_0+d_0)x-b_0}{-\alpha_{A_0}(x+\sigma_{A_0})}+\rho_{A_1}
=
\frac{ (b_0+d_0-\alpha_{A_0}\rho_{A_1})x +(a_0-b_0)\rho_{A_1}-b_0}{-\alpha_{A_0}(x+\sigma_{A_0})}
$$
so (\ref{suffinc}) is seen to be the M\"obius function
\begin{equation*}
x\mapsto \frac{x+\rho_{A_1}}{ (b_0+d_0-\alpha_{A_0}\rho_{A_1})x +(a_0-b_0)\rho_{A_1}-b_0} \,,
\end{equation*}
which is known to be strictly increasing by Lemma \ref{techderivlemma}.
\end{proof}

As a consequence of Theorem \ref{t1smaller1theorem}
we obtain:

\begin{cor}\label{t1smallerttheorem}
If $\A\in\mm$ and $t\in \mathbb{R}^+$ are such that
\begin{equation}\label{t1smallert}
t \ge t_1(\A) \,,
\end{equation}
then the Dirac measure at the fixed point $p_{A_1}$ is the unique $f_{\A(t)}$-maximizing measure;
in particular, the joint spectral radius of $\A(t)$ is equal to the spectral radius of $tA_1$.
\end{cor}
\begin{proof}
The assumption (\ref{t1smallert}) means, using (\ref{tti}),
that $t_1(\A(t))\le 1$, so the result follows by applying
Theorem \ref{t1smaller1theorem} 
with $\A$ replaced by $\A(t)$.
\end{proof}

\section{Sturmian maximizing measures}\label{technicalsection}

It is at this point that we make the extra hypothesis 
that the matrix pair $\A$ lies in the class $\nn \subset \mm$.
By Lemma \ref{posnegf}(ii) we know that if $\A\in\mm$ then $f_\A$ is strictly increasing on $X_{A_0}$ and strictly decreasing on $X_{A_1}$; the following result asserts that
if we make the stronger hypothesis that $\A\in\nn$ then
 these monotonicity properties
are inherited by all functions formed by adding a Sturmian transfer function $\phi_\Gamma$  to $f_\A$.

\begin{prop}\label{increasingdecreasingprop}
Let $\A\in\nn$.
For each $\A$-Sturmian interval $\Gamma\in \ss_\A$, 
the function $f_\A+\varphi_\Gamma:X_\A\to\mathbb{R}$ is strictly increasing on $X_{A_0}$,
and strictly decreasing on $X_{A_1}$.
\end{prop}
\begin{proof}
First suppose $x\in X_{A_0}$.
Let $0=i_0<i_1<i_2<\ldots$ be the sequence of all
integers such that $\tau_\Gamma^{i_k}(x) \in X_{A_0}$.

For $k\ge0$, writing $z=\tau_\Gamma^{i_k}(x)$ we see that
if $1\le i<i_{k+1}-i_k$ then
$\tau_\Gamma^i(z)\in X_{A_1}$, 
and thus
$\tau_\Gamma^i(z)=T_{A_1}^i(z)$, so that
\begin{equation}\label{tauta1pos}
\sum_{i=0}^{i_{k+1}-i_k-1} (f_\A\circ \tau_\Gamma^i)'(z)
= 
f_\A'(z) + \sum_{i=1}^{i_{k+1}-i_k-1} (f_\A\circ T_{A_1}^i)'(z) 
>
f_\A'(z) + \sum_{i=1}^{\infty} (f_\A\circ T_{A_1}^i)'(z)\,,
\end{equation}
where the inequality is because
$(f_\A\circ T_{A_1}^i)'(z) <0$ for all $i\ge1$, by Lemma \ref{posnegf}.
Now 
$z\in X_{A_0}$, so
(\ref{fderivative}) in Lemma \ref{simple} (iii)
gives
$f_\A' (z)= -(z+\sigma_{A_0})^{-1}$ (which is positive),
and formula
(\ref{sumformula}) from Corollary \ref{usefulcor} gives
$\sum_{i=1}^\infty (f_\A\circ T_{A_1}^i)'(z) = (z+\varrho_{A_1})^{-1}$ (which is negative), so
(\ref{tauta1pos}) implies that
\begin{equation}\label{sigrho}
\sum_{i=0}^{i_{k+1}-i_k-1} (f_\A\circ \tau_\Gamma^i)'(z)
>
\frac{-1}{z+\sigma_{A_0}} + \frac{1}{z+\varrho_{A_1}} \,.
\end{equation}
However $\A\in\nn$, so $\rho_{A_1}<\sigma_{A_0}$, and therefore the righthand side of
(\ref{sigrho}) is positive, so we have shown that
\begin{equation*}
\sum_{i=0}^{i_{k+1}-i_k-1} (f_\A\circ \tau_\Gamma^i)'(z)
>
0\,.
\end{equation*}

It follows that for all $k\ge0$,
$$
\sum_{n=i_k}^{i_{k+1}-1} (f_\A\circ \tau_\Gamma^n)'(x)
=
(\tau_\Gamma^{i_k})'(x) \sum_{i=0}^{i_{k+1}-i_k-1} (f_\A\circ \tau_\Gamma^i)'(z) > 0\,,
$$
and hence
$$
(f_\A+\phi_\Gamma)'(x)
=\sum_{n=0}^\infty (f_\A\circ \tau_\Gamma^n)'(x)
=\sum_{k=0}^\infty \sum_{n=j_k}^{j_{k+1}-1} (f_\A\circ \tau_\Gamma^n)'(x) >0\,,
$$
so $f_\A+\phi_\Gamma$ is strictly increasing on $X_{A_0}$.

Now suppose $x\in X_{A_1}$. The proof proceeds analogously to the above.
Let $0=j_0<j_1<j_2<\ldots$ be the sequence of all
integers such that $\tau_\Gamma^{j_k}(x) \in X_{A_1}$.

For $k\ge0$, writing $z=\tau_\Gamma^{j_k}(x)$ we see that
if $1\le i<j_{k+1}-j_k$ then
$\tau_\Gamma^i(z)\in X_{A_0}$, 
and thus
$\tau_\Gamma^i(z)=T_{A_0}^i(z)$, so that
\begin{equation}\label{sigrho2}
\sum_{i=0}^{j_{k+1}-j_k-1} (f_\A\circ \tau_\Gamma^i)'(z)
=
f_\A'(z) + \sum_{i=1}^{j_{k+1}-j_k-1} (f_\A\circ T_{A_0}^i)'(z) 
 <
f_\A'(z) + \sum_{i=1}^{\infty} (f_\A\circ T_{A_0}^i)'(z) \,,
\end{equation}
using the fact that
$(f_\A\circ T_{A_0}^i)'(z) >0$ for all $i\ge1$, by Lemma \ref{posnegf}. 
The righthand side of (\ref{sigrho2}) can be written as
$-(z+\sigma_{A_1})^{-1} + (z+\varrho_{A_0})^{-1}$
using Lemma \ref{simple} (iii)
and
Corollary \ref{usefulcor},
and this is strictly negative since $\sigma_{A_1}<\varrho_{A_0}$
because $\A\in\nn$,
so we have shown that
\begin{equation*}
\sum_{i=0}^{j_{k+1}-j_k-1} (f_\A\circ \tau_\Gamma^i)'(z)
<
\frac{-1}{z+\sigma_{A_1}} + \frac{1}{z+\varrho_{A_0}} < 0\,.
\end{equation*}
It follows that for all $k\ge0$,
$$
\sum_{n=j_k}^{j_{k+1}-1} (f_\A\circ \tau_\Gamma^n)'(x)
=
(\tau_\Gamma^{j_k})'(x) \sum_{i=0}^{j_{k+1}-j_k-1} (f_\A\circ \tau_\Gamma^i)'(z) < 0\,,
$$
and hence
$$
(f_\A+\phi_\Gamma)'(x)
=\sum_{n=0}^\infty (f_\A\circ \tau_\Gamma^n)'(x)
=\sum_{k=0}^\infty \sum_{n=j_k}^{j_{k+1}-1} (f_\A\circ \tau_\Gamma^n)'(x) <0\,,
$$
so $f_\A+\phi_\Gamma$ is strictly decreasing on $X_{A_1}$.
\end{proof}

\begin{theorem}\label{thm6}
Let $\A\in\nn$ and $t\in \t_\A= (t_0(\A),t_1(\A))$.
The $\A$-Sturmian measure 
supported by
the $\A$-Sturmian interval $\Gamma_\A(t)$ is the unique maximizing measure
for $f_{\A(t)}$; thus the corresponding Sturmian measure on $\Omega = \{0,1\}^\N$ is the unique $\A(t)$-maximizing measure.
\end{theorem}
\begin{proof}
Let us write $\varphi = \varphi_{\Gamma_\A(t)}$ and $T=T_\A=T_{\A(t)}$.
We know that
 $f_{\A(t)}+\varphi-\varphi\circ T$
is a constant function when restricted to $\Gamma_\A(t)=: [\gamma_t^-,\gamma_t^+] \cap X_\A$,
by Proposition \ref{analogue}.
In particular, 
$$(f_{\A(t)}+\varphi-\varphi\circ T)\left(\gamma_t^-\right)
=
(f_{\A(t)}+\varphi-\varphi\circ T)\left(\gamma_t^+\right)\,,
$$
and because 
$T(\gamma_t^-)=T(\gamma_t^+)$,
we deduce that
\begin{equation}\label{equality}
(f_{\A(t)}+\varphi)\left(\gamma_t^-\right)
=
(f_{\A(t)}+\varphi)\left(\gamma_t^+\right)
\,.
\end{equation}

But Proposition \ref{increasingdecreasingprop} implies that
$f_{\A(t)}+\varphi$ is strictly increasing on $X_{A_0}$,
and strictly decreasing on $X_{A_1}$, so together with (\ref{equality}) we deduce that
\begin{equation}\label{strongsturmian}
(f_{\A(t)}+\varphi)(x) > (f_{\A(t)}+\varphi)(y)\quad\text{for all }x\in \Gamma_\A(t),\ y\in X_\A \setminus \Gamma_\A(t)\,.
\end{equation}

Consequently, if $z,z'$ are such that $T(z)=T(z')$, with $z\in \Gamma_\A(t)$ and $z'\notin \Gamma_\A(t)$, then
$$
(f_{\A(t)}+\varphi)(z) > (f_{\A(t)}+\varphi)(z')\,,
$$
and hence
$$
(f_{\A(t)}+\varphi-\varphi\circ T)(z) > (f_{\A(t)}+\varphi-\varphi\circ T)(z')\,.
$$
In other words, the constant value of
$f_{\A(t)}+\varphi-\varphi\circ T$ on $\Gamma_\A(t)$ is its global maximum,
and this value is not attained at any point in $X_\A \setminus \Gamma_\A(t)$.

It follows that the Sturmian measure supported by $\Gamma_\A(t)$ is the unique maximizing measure
for $f_{\A(t)}+\varphi-\varphi\circ T$,
and hence the unique maximizing measure for $f_{\A(t)}$.
Thus the corresponding Sturmian measure on $\Omega=\{0,1\}^\N$ is the unique $\A(t)$-maximizing measure.
\end{proof}

Recall from \S \ref{generalsection} that
$\ee \subset M_2(\R)^2$
denotes the set of matrix pairs which are equivalent to some pair in $\nn$, where
equivalence of
$\A=(A_0,A_1)$ and $\A'=(A_0',A_1')$ means that
$A_0'=uP^{-1}A_0P$ and $A_1'=vP^{-1}A_1P$ for some invertible $P$ and $u,v>0$.
We deduce the following theorem:

\begin{theorem}\label{deducedthm}
If $\A\in\ee$ and $t\in\mathbb{R}^+$, then $\A(t)$ has a unique maximizing measure, and this maximizing measure is Sturmian. 
\end{theorem}
\begin{proof}
It suffices to prove the result for $\A\in\nn$, and this is
immediate from Corollaries \ref{t0largerttheorem} and \ref{t1smallerttheorem},
Theorem \ref{thm6}, and
 the fact that $\nn \subset \mm$.
\end{proof}

\section{The parameter map is a devil's staircase}\label{devilsection}

As noted in Remark \ref{conjugacymeasures}, if $\A\in\mm$ then there is a topological conjugacy
$h_\A:\Omega\to Y_\A$ between the the shift map $\sigma:\Omega\to\Omega$ 
and the restriction of $T_\A$ to the Cantor set $Y_\A\subset X_\A$; the map $h_\A$ is strictly increasing with respect to the orders on $\Omega$ and $Y_\A$ (cf.~Remark \ref{sturmianasturmian}).
If $d:\Omega\to [0,1]$ is as in Proposition \ref{sturmianclassical}\,(c), associating to 
$\omega\in\Omega$ the Sturmian parameter of the measure supported by $[0\omega,1\omega]$,
then the map $d_\A:Y_\A\to[0,1]$ given by
 $d_\A=d\circ h_\A^{-1}$
 enjoys the same properties as $d$:

\begin{lemma}
The map $d_\A: Y_\A\to [0,1]$
is continuous, non-decreasing, and surjective. 
The preimage $d_\A^{-1}(\mathcal{P})$ is a singleton if $\mathcal{P}$ is irrational,
and a positive-length closed interval if $\mathcal{P}$ is rational.
\end{lemma}
\begin{proof}
Immediate from Proposition \ref{sturmianclassical}\,(c), and the fact that $h_\A$ is strictly increasing.
\end{proof}

Note that $d_\A$ associates to $y\in Y_\A$
the parameter of the $\A$-Sturmian measure supported by the $\A$-Sturmian
interval
 $c_\A^{-1}(y)$,
 where
we recall from Definition \ref{henceforthca} that the identification map $c_\A:\ss_\A\to [0,1]$ is defined by
$c_\A(\Gamma)=T_\A(\min \Gamma) = T_\A(\max \Gamma)$.
Of the extensions of the function $d_\A$ from the Cantor set $Y_\A$ to the interval $X=[0,1]$,
there is a unique one giving a non-decreasing self-map $d_\A:X\to [0,1]$.
This extension, which we shall also denote by $d_\A$, is continuous,  and $d_\A(c)$ is just the parameter of the $\A$-Sturmian measure
$s_{c_\A^{-1}(c)}$ (i.e.~of the $\A$-Sturmian measure supported by
the $\A$-Sturmian interval $c_\A^{-1}(c)$) for each $c\in X$.
We therefore have the following:

\begin{cor}\label{dadev}
The map $d_\A: X\to [0,1]$
is continuous, non-decreasing, and surjective. 
The preimage $d_\A^{-1}(\mathcal{P})$ is a singleton if $\mathcal{P}$ is irrational,
and a positive-length closed interval if $\mathcal{P}$ is rational.
\end{cor}

\begin{defn}
For $\A\in\nn$, let $\mathcal{P}_\A(t)$ denote the parameter of the Sturmian maximizing measure for $\A(t)$,
or equivalently of the $\A$-Sturmian $f_{\A(t)}$-maximizing measure.
This defines the \emph{parameter map} $\mathcal{P}_\A:\mathbb{R}^+\to[0,1]$.
\end{defn}
Recalling 
(see Proposition \ref{analogue})
the map $t\mapsto \Gamma_\A(t)$ associating $\A$-Sturmian interval to parameter 
$t\in\t_\A =  (t_0(\A),t_1(\A))$,
we see that in fact the map $\mathcal{P}_\A:\t_\A\to X$ can be written as
\begin{equation}\label{rfactor}
\mathcal{P}_\A = d_\A\circ c_\A\circ \Gamma_\A\,.
\end{equation}
This means that $\mathcal{P}_\A$ will enjoy the same properties as established for $d_\A$ in Corollary
\ref{dadev}, provided $c_\A\circ \Gamma_\A$ is strictly increasing:

\begin{lemma}\label{strictinccomp}
For $\A\in\mm$, the map $c_\A\circ \Gamma_\A: \t_\A\to X$ is strictly increasing and surjective.
\end{lemma}
\begin{proof}
Recall from
(\ref{gadef})
the function $G_\A$ given by
$$G_\A(t) = \log \left( \left(\frac{a_0+c_0}{b_1+d_1}\right) t^{-1} \right)\,,$$
and that $\Gamma_\A(t)\in \ss_A$ is defined
(see (\ref{deltatdefeq})) by the identity
\begin{equation*}
 \Delta_\A \circ \Gamma_\A
 = G_\A
  \,.
\end{equation*}
Now $G_\A$ is strictly decreasing, so in particular injective,
therefore the map $\Gamma_\A$ is necessarily injective. 
Note that $\Gamma_\A$ clearly extends to a continuous injection on
$\overline{\t_\A}= [t_0(\A),t_1(\A)]$, with $\Gamma_\A(t_i(\A))=\Gamma_i$ for $i\in\{0,1\}$.

Now $c_\A:\ss_\A\to X$ is a bijection,
so $c_\A\circ \Gamma_\A:\overline{\t_\A}\to X$ is injective, and its continuity means it is strictly monotone.
But
$c_\A(\Gamma_\A(t_0(\A)))=0$
and $c_\A(\Gamma_\A(t_1(\A)))=1$, so the map
$c_\A\circ\Gamma_\A$ must be strictly increasing and surjective, as required.
\end{proof}

We can now prove that the parameter map $\mathcal{P}_\A:\mathbb{R}^+\to[0,1]$ is 
singular. More specifically, its properties described by the following Theorem \ref{deviltheorem}
mean it is a \emph{devil's staircase}. These properties of the parameter map had been noted
by Bousch \& Mairesse \cite{bouschmairesse} in the context of the family (\ref{bmfamily}),
and proved in detail by Morris \& Sidorov \cite{morrissidorov} for the family (\ref{standardpair}).
The following result can be viewed as a 
more detailed version of Theorem
\ref{maintheorem} from \S \ref{generalsection}:

\begin{theorem}\label{deviltheorem}
If $\A\in\ee$ and $t\in\mathbb{R}^+$, then $\A(t)$ has a unique maximizing measure, and this maximizing measure is Sturmian. 
Let $\mathcal{P}_\A(t)$ denote the parameter of the Sturmian maximizing measure for $\A(t)$.
The parameter map $\mathcal{P}_\A:\mathbb{R}^+\to [0,1]$ 
is continuous, non-decreasing, and surjective. 
The preimage $\mathcal{P}_\A^{-1}(\mathcal{P})$ is a singleton if $\mathcal{P}$ is irrational,
and a positive-length closed interval if $\mathcal{P}$ is rational.
\end{theorem}
\begin{proof}
The set $\ee$ consists of matrix pairs which are equivalent to a matrix pair in $\nn$, so it suffices to prove the result
for $\A\in\nn$. 
Theorem \ref{deducedthm} gives that
$\A(t)$ has a unique maximizing measure, and that this maximizing measure is Sturmian.

For $t\in \mathbb{R}^+\setminus \t_\A$ we know that
\begin{equation}\label{rot0}
\mathcal{P}_\A(t)=0 \quad\text{for }t\in (0,t_0(\A))
\end{equation}
by Theorem \ref{t0largerttheorem},
and
\begin{equation}\label{rot1}
\mathcal{P}_\A(t)=1 \quad\text{for }t\in (t_1(\A),\infty) 
\end{equation}
by Theorem \ref{t1smallerttheorem}, since the Dirac measures at the fixed points
$p_{A_0}$ and $p_{A_1}$ are $\A$-Sturmian measures of parameters 0 and 1 respectively.

In view of (\ref{rot0}) and (\ref{rot1}), it suffices to 
establish the required properties of $\mathcal{P}_\A$ on the sub-interval $\t_\A=(t_0(\A),t_1(\A))$.
Using
the factorisation
(\ref{rfactor}),
we see that this follows 
from Corollary \ref{dadev}
and Lemma \ref{strictinccomp}.
\end{proof}


\begin{thebibliography}{1}

\bibitem{blondel} V. D. Blondel,
The birth of the joint spectral radius: an interview with Gilbert Strang,
{\it  Linear Algebra Appl.}, {\bf 428} (2008), 2261--2264.


\bibitem{btv}
V. D. Blondel, J. Theys \& A. A. Vladimirov, An elementary counterexample to the finiteness conjecture, 
{\it SIAM Journal on Matrix Analysis}, {\bf 24}
(2003),  963Ð970.

\bibitem{bochirams} J. Bochi \& M. Rams,
The entropy of Lyapunov-optimizing measures of some matrix cocycles,
{\it preprint}, arxiv:1312.6718

\bibitem{bousch} T. Bousch,  
Le poisson n'a pas
  d'ar\^{e}tes, {\it Ann. Inst. Henri Poincar\'e (Proba. et Stat.)} {\bf
    36}, (2000), 489--508.

\bibitem{bouschmairesse}
T. Bousch \& J. Mairesse,
Asymptotic height optimization for topical IFS, Tetris heaps, and the finiteness conjecture,
{\it  J. Amer. Math. Soc.}, {\bf 15} (2002), 77--111.



\bibitem{bullettsentenac}
S. Bullett \& P. Sentenac,
Ordered orbits of the shift, square roots, and the
devil's staircase,
{\it Math.\ Proc.\ Camb.\ Phil.\ Soc.},
{\bf 115} (1994),
451--481.


\bibitem{daubechieslagarias}
I. Daubechies \& J. C. Lagarias,
 Sets of matrices all infinite products of which converge,
 {\it  Linear Algebra Appl.}, {\bf 162} (1992) 227-261

\bibitem{gurvits}
L. Gurvits,
Stability of Linear Inclusions--Part 2, 
NECI Technical Report TR pp.96Ð173, 1996.

\bibitem{hmst}
K.~G.~Hare, I.~ D.~Morris, N.~Sidorov,  \& J.~Theys, An explicit counterexample to the 
Lagarias-Wang finiteness conjecture, {\it Adv. Math.}, {\bf 226} (2011),
4667--4701.


\bibitem{jeo}
O. Jenkinson, Ergodic optimization,
{\it Discrete \& Cont. Dyn. Sys.}, {\bf 15} (2006), 197--224.

\bibitem{jungers} R. Jungers, {\it The joint spectral radius},
vol.~385 of Lecture Notes in Control and Information Sciences,
Springer-Verlag, Berlin, 2009.

\bibitem{kk}
R. Kannan \& C. K.~Krueger,
{\it Advanced analysis on the real line},
Springer-Verlag, New York, 1996.


\bibitem{kozyakin}
V.~S.~Kozyakin,
A dynamical systems construction of a counterexample to the finiteness conjecture,
{\it in Proceedings of the 44th IEEE Conference on Decision and Control, and the European Control Conference 2005, Seville, Spain, December 2005}, pp.~2338--2343.

\bibitem{kozyakinbiblio}
V. S. Kozyakin,
An annotated bibliography on convergence of matrix products and the theory of joint/generalized spectral radius, {\it preprint}.

\bibitem{lagariaswang}
J. C. Lagarias \& Y. Wang,
The finiteness conjecture for the generalized spectral radius of a set of matrices,
{\it Linear Algebra Appl.}, 214:17Ð42, 1995.

\bibitem{maesumi} M. Maesumi,
Optimal norms and the computation of joint spectral radius of matrices,
{\it Linear Algebra Appl.}, {\bf 428} (2008), 2324--2338.

\bibitem{morrissidorov}
I.~D.~Morris \& N.~Sidorov,
On a devil's staircase associated to the joint spectral radii of a family of pairs of matrices,
{\it J. Eur. Math. Soc.}, {\bf 15} (2013), 1747--1782. 



\bibitem{morsehedlund} M. Morse and G. A. Hedlund, Symbolic
  Dynamics II. Sturmian Trajectories, {\it Amer. J. Math.}, {\bf 62}
  (1940), 1--42.

\bibitem{parthasarathy}
K. R. Parthasarathy,  On the category of ergodic measures, 
{\it Illinois J. Math.}, {\bf 5} (1961), 648--656.

\bibitem{rotastrang}
G-C. Rota \& G. Strang,  A note on the joint spectral radius,
{\it  Indag. Math.} 22 (1960) 379-381


\bibitem{sigmund}
K. Sigmund,
Generic properties of invariant measures for Axiom $A$
diffeomorphisms,
{\em Invent. Math.},  {\bf 11} (1970), 99--109.


\bibitem{strang} G. Strang, {\it The joint spectral radius},
Commentary by Gilbert Strang on paper number 5, in Collected Works of Gian-Carlo Rota, 2001;
available online from http://www-math.mit.edu/$\sim$gs


\bibitem{walters}
P.~Walters, {\it An introduction to ergodic theory}, Springer, 1981.

\end{thebibliography}
\end{document}